\documentclass[11pt,reqno,twoside]{article}

\usepackage[hang]{footmisc}
\usepackage{lipsum}

\usepackage{cmap} % Load before fontenc 

\usepackage[T1]{fontenc}
\usepackage[utf8]{inputenc}
\usepackage{graphicx}
\usepackage{array}
\usepackage{placeins}
\usepackage{enumerate}
\usepackage{algcompatible}
\usepackage[most]{tcolorbox}
\usepackage{mdframed}
\usepackage{verbatim}
\newcommand{\comments}[1]{}

\usepackage{soul}

\usepackage{setspace}

\let\counterwithin\relax  %DSA: i had to include this to be able to compile
\usepackage{lmodern} % Improved version of computer modern

\usepackage{comment}

\usepackage{bm} % boldmath must be called after the package

\usepackage{bbold}

\usepackage{amsmath,amsbsy,amsgen,amscd,amsthm,amsfonts,amssymb}

\usepackage{float}       
\usepackage{algorithm,algpseudocode}
\algrenewcommand\algorithmicrequire{\textbf{Input:}} \algrenewcommand\algorithmicensure{\textbf{Output:}}

\usepackage[centering,top=1.1in,bottom=1.4in,left=1in,right=0.9in]{geometry}

\usepackage[sf,bf,compact]{titlesec}
\usepackage[mathscr]{euscript}

\newcommand{\mcO}{\mathcal{O}}
\newcommand{\mcG}{\mathcal{G}}
\newcommand{\mcA}{\mathcal{A}}

\newcommand{\mcE}{\mathcal{E}}

\newcommand{\mcR}{\mathcal{R}}

\usepackage{enumitem}
\usepackage[dvipsnames]{xcolor}

\definecolor{dark-gray}{gray}{0.3}
\definecolor{dkgray}{rgb}{.4,.4,.4}
\definecolor{dkblue}{rgb}{0,0,.5}
\definecolor{medblue}{rgb}{0,0,.75}
\definecolor{rust}{rgb}{0.5,0.1,0.1}

\usepackage{url}
\usepackage[colorlinks=true]{hyperref}
\hypersetup{linkcolor=dkblue}    
\hypersetup{citecolor=rust}      
\hypersetup{urlcolor=rust}     

\usepackage[final]{microtype} 

\newtheoremstyle{myThm} % name
    {\topsep}                    % Space above
    {\topsep}                    % Space below
    {\itshape}                   % Body font
    {}                           % Indent amount
    {\sffamily\bfseries}                   % Theorem head font
    {.}                          % Punctuation after theorem head
    {.5em}                       % Space after theorem head
    {}  % Theorem head spec (can be left empty, meaning ‘normal’)

\newtheoremstyle{myRem} % name
    {\topsep}                    % Space above
    {\topsep}                    % Space below
    {}                   % Body font
    {}                           % Indent amount
    {\sffamily}                   % Theorem head font
    {.}                          % Punctuation after theorem head
    {.5em}                       % Space after theorem head
    {}  % Theorem head spec (can be left empty, meaning ‘normal’)

\newtheoremstyle{myDef} % name
    {\topsep}                    % Space above
    {\topsep}                    % Space below
    {}                   % Body font
    {}                           % Indent amount
    {\sffamily\bfseries}                   % Theorem head font
    {.}                          % Punctuation after theorem head
    {.5em}                       % Space after theorem head
    {}  % Theorem head spec (can be left empty, meaning ‘normal’)

\theoremstyle{myThm}
\newtheorem{theorem}{Theorem}[section]
\newtheorem{lemma}[theorem]{Lemma}
\newtheorem{proposition}[theorem]{Proposition}

\newtheorem{definition}[theorem]{Definition}

\theoremstyle{myRem}

\theoremstyle{myRem}
\newtheorem{remarkx}[theorem]{Remark}
 \newenvironment{remark}
  {\pushQED{\qed}\remarkx}
  {\popQED\endremarkx}
\usepackage{fancyhdr}
\let\originalleft\left
\let\originalright\right
\renewcommand{\left}{\mathopen{}\mathclose\bgroup\originalleft}
\renewcommand{\right}{\aftergroup\egroup\originalright}

\usepackage{mathtools}
\mathtoolsset{centercolon}  % Makes := typeset correctly for definitions

\newcommand{\circnum}[1]{\raisebox{.5pt}{\textcircled{\raisebox{-.2pt}{\scriptsize #1}}}}

\providecommand{\mathbbm}{\mathbb} % In case we don't load bbm

\newcommand{\R}{\mathbbm{R}}
\newcommand{\E}{\mathbbm{E}}

\renewcommand{\mcR}{\mathcal{R}}

\newcommand{\mcM}{\mathcal{M}}
\newcommand{\mcB}{\mathcal{B}}

\newcommand{\bfF}{\mathbf{F}}
\newcommand{\bfU}{\mathbf{U}}
\newcommand{\bfW}{\mathbf{W}}
\newcommand{\bfA}{\mathbf{A}}

\newcommand{\bfB}{\mathbf{B}}
\newcommand{\bfC}{\mathbf{C}}
\newcommand{\bfD}{\mathbf{D}}

\newcommand{\bfG}{\mathbf{G}}

\newcommand{\bfZ}{\mathbf{Z}}
\newcommand{\bfu}{\mathbf{u}}
\newcommand{\bff}{\mathbf{f}}
\newcommand{\bfh}{\mathbf{h}}
\newcommand{\bfw}{\mathbf{w}}

\newcommand{\bfy}{\mathbf{y}}
\newcommand{\bfx}{\mathbf{x}}

\newcommand{\mcN}{\mathcal{N}}
\newcommand{\mcS}{\mathcal{S}}
\newcommand{\mcL}{\mathcal{L}}
\newcommand{\mcH}{\mathcal{H}}

\newcommand{\tr}{\mathrm{Tr}}   % trace operator: \tr(A)
\newcommand{\Cost}{\mathsf{Cost}}
\newcommand{\ErrorOVB}{\mathsf{Error\text{-}OVB}}
\newcommand{\Var}{\mathsf{Var}}

\newcommand{\ErrorVar}{\mathsf{Error\text{-}Var}}
\newcommand{\ErrorBias}{\mathsf{Error\text{-}Bias}}

\newcommand{\ErrorEstimation}{\mathsf{Error\text{-}Estimation}}

\definecolor{mygreen}{rgb}{0.13,0.55,0.13}

\newcommand{\Prob}{\operatorname{\mathbbm{P}}}

\newcommand{\iid}{\stackrel{\text{i.i.d.}}{\sim}}

\newcommand{\dkl}{D_{\mbox {\tiny{\rm KL}}}}

\usepackage[font = small, margin=30pt]{caption}
\usepackage[]{algorithm}
\usepackage{enumerate}

\usepackage{graphicx}

\usepackage{authblk}
\usepackage{chngcntr}
\usepackage{mathrsfs} 
\counterwithin{table}{section}
\counterwithin{algorithm}{section}

\usepackage[normalem]{ulem}

\usepackage{tikz}
\usetikzlibrary{shapes,arrows,arrows.meta, decorations.pathmorphing,backgrounds,positioning,fit,petri}
\usetikzlibrary{calc}
\usetikzlibrary{matrix}
\usetikzlibrary{backgrounds}
\usetikzlibrary{shapes.geometric}
\usetikzlibrary{patterns}
\usepackage{pgfplots}
\pgfplotsset{compat=newest}
\usepgfplotslibrary{groupplots}
\pgfdeclarelayer{background}
\pgfdeclarelayer{bBack}
\pgfsetlayers{bBack,background,main}
\usetikzlibrary{fit}

\usepackage{subcaption}

\title{\huge Optimal Multiscale Learning of Linear Operators}

\author{Jiaheng Chen and Daniel Sanz-Alonso}

\date{University of Chicago}

\vspace{.5in}

\makeatletter\@addtoreset{section}{part}\makeatother%
\numberwithin{equation}{section}

\newcommand{\upperRomannumeral}[1]{\uppercase\expandafter{\romannumeral#1}}

\begin{document}
\maketitle %  LEAVE HERE
% The command above causes the title to be displayed.

%>>>>> DELETE ALL CONTENT UNTIL "\end{document}"
% This is the body of your document.

\renewcommand{\thefootnote}{\fnsymbol{footnote}}

% \footnotetext[1]{Committee on Computational and Applied Mathematics, University of Chicago}
% \footnotetext[2]{Department of Statistics, University of Chicago}

\vspace{-2em}
\abstract{We study the statistical and computational limits of learning bounded linear operators between Sobolev spaces from noisy input--output data. In wavelet coordinates, the problem is recast as an infinite-dimensional matrix regression problem with a heterogeneous two-sided multiscale structure. We establish minimax rates under Sobolev operator-norm loss and construct a finite-resolution blockwise
least-squares estimator attaining these rates.
The analysis reveals a nonuniform local estimation difficulty across scales, which can be exploited algorithmically: by assigning scale-adaptive sample sizes, the estimator achieves the optimal computational cost among dense least-squares implementations.
}


\bigskip 

\section{Introduction}

Operator learning is an emerging paradigm at the intersection of scientific computing and machine learning \cite{li2020fourier,anandkumar2020neural,lu2021deepxde,lu2021learning,pathak2022fourcastnet,kovachki2023neural,li2024physics,boulle2024mathematical,kovachki2024operator,subedi2025operator}. Its central objective is to learn maps between function spaces from data,  a setting that arises, for instance, in the approximation of solution operators, parameter-to-observable maps, and time-evolution maps for dynamical systems. In a supervised learning formulation, one observes noisy input--output pairs $\{(u_i,f_i)\}_{i=1}^{N}$ satisfying
\begin{align}\label{eq:intro_continuous_model}
f_i = \mcA u_i + w_i, \qquad 1 \le i \le N,
\end{align}
where $\mcA$ is an unknown operator and $\{w_i\}_{i=1}^{N}$ models observational noise. The goal is to design an estimator of $\mcA$ that is statistically accurate and computationally efficient.

This paper studies this problem for bounded linear operators between Sobolev
spaces on a $d$-dimensional domain. For fixed Sobolev indices $s,s'$ and radius $R>0$, we consider the operator class
\[
\mcO(s,s',R)
:=
\bigl\{
\mcA\in \mcL(H^s,H^{-s'}):
\|\mcA\|_{H^s\to H^{-s'}}\le R
\bigr\}.
\]
The inputs and noise are Gaussian fields,
\(u_i\sim \mathcal{GP}(0,(I-\Delta)^{-r_1})\) and
\(w_i\sim \mathcal{GP}(0,(I-\Delta)^{-r_2})\). The estimation error is
measured in the Sobolev operator norm \(H^t\to H^{-t'}\), where \(t>s\) and
\(t'>s'\). This model imposes no PDE-specific, geometric, or sparsity structure on the target operator. Instead, the problem is governed by the interaction between the Sobolev scales of the operator class, the input distribution, the noise model, and the loss norm; this structure alone is sufficient to obtain polynomial statistical and computational rates.

Our first main result, Theorem~\ref{thm:main1}, establishes the minimax rate up
to logarithmic factors:
\begin{align*}
\inf_{\widehat{\mcA}}\sup_{\mcA\in \mcO(s,s',R)} \E\, \|\widehat{\mcA}-\mcA\|_{H^{t}\to H^{-t'}} \, \asymp_{\log} \, N^{-\min\big\{\frac{1}{2},\frac{t-s}{2(r_1-s)+d},\frac{t'-s'}{(2(-r_2-s')+d)_{+}}\big\}}.
\end{align*}
The three terms in the exponent correspond to a coarse-scale parametric
contribution, an input-side obstruction caused by weak excitation of
high-frequency input directions, and an output-side obstruction caused by estimating rough output components from noisy observations.
Thus the rate reflects a
two-sided multiscale bias--variance geometry.

The upper bound is constructive. In wavelet coordinates, the problem becomes an
infinite-dimensional matrix regression problem with a diagonally weighted
operator-norm loss. 
We construct a finite-resolution blockwise least-squares estimator that attains the minimax rate while estimating a minimal collection of wavelet blocks. A nested-support regression step controls the omitted-variable bias arising from interactions between
retained and discarded input scales. 
The estimator requires only finite-resolution observations of the training input--output pairs, with the required resolution determined by the sample size. The minimax lower bound is obtained by applying Assouad's lemma to three thin-strip
perturbation families in the wavelet matrix model. These perturbations isolate
the coarse-scale, input-side, and output-side obstructions, respectively, and
show that each term in the minimax exponent is unavoidable.

Our second main contribution is computational. The local estimation difficulty
is highly nonuniform across scales: some blocks determine the global minimax
rate, while others can be estimated to the same target accuracy using fewer
samples. We exploit this inhomogeneity through scale-adaptive sample sizes.
Proposition~\ref{prop:cost} shows that, if the target accuracy is
\(\varepsilon\), then the cost of our estimator satisfies
\[
\Cost(\varepsilon)
\lesssim_{\log}
\varepsilon^{-\max\big\{
2
+
\frac{2d}{\min\{r_1,t\}-s}
+
\frac{(2(r_1-t)+d)_{+}}{t-s},
\ 
\frac{d+(2(-r_2-s')+d)_{+}}{t'-s'}
\big\}}.
\]
The two terms in this maximum correspond to input-side and output-side
computational contributions. We further explain why these contributions are
unavoidable within the dense least-squares framework. In addition, we show that the statistical and computational bottlenecks need not coincide, highlighting that minimax rates and computational cost provide complementary perspectives on the operator learning problem.

The rest of the paper is organized as follows. The remainder of the
introduction reviews related work and sets notation. Section~\ref{sec:main_results} states and discusses the main results. Section~\ref{sec:main_idea} develops the four main ideas underlying our analysis and estimator construction: the wavelet-coordinate matrix formulation, the two-sided
bias--variance geometry, the identification of a minimal collection of wavelet blocks sufficient for attaining the minimax rate, and scale-adaptive
sample usage. Section~\ref{sec:convergencerate} constructs the estimator and
proves the minimax upper and lower bounds. Section~\ref{sec:computation}
analyzes the computational cost and discusses its optimality. Appendix~\ref{app:convergencerate} contains auxiliary estimates used in the
proofs.

\subsection{Related work}

We focus on the works most closely related to the present paper, and refer the
reader to recent surveys for broader perspectives on operator learning
\cite{boulle2024mathematical,kovachki2024operator,subedi2025operator,
brugiapaglia2026short}.

A growing body of work studies approximation, statistical, and computational
complexity for learning operators between function spaces. General complexity
results for broad Lipschitz-type classes reveal intrinsic limitations and curse
of dimensionality phenomena
\cite{schwab2023deep,adcock2024sample,kovachki2024data,
lanthaler2024operator,adcock2025towards}. On the other hand, more structured
classes, such as holomorphic operators and holomorphic maps arising in
parametric PDEs, admit algebraic approximation and recovery rates under
anisotropy, summability, or source-type assumptions
\cite{cohen2015approximation,herrmann2022neural,adcock2024optimal,
adcock2025optimal,reinhardt2024statistical}. Our results identify a
complementary polynomial-rate regime: bounded linear operators between Sobolev
spaces form a broad infinite-dimensional nonparametric class, but still admit
polynomial minimax rates under Sobolev operator-norm loss. These rates are
obtained from linearity and Sobolev-scale boundedness, without holomorphic,
source-type, or sparsity assumptions on the target operator.

Closer to our setting are works on learning linear operators and related
infinite-dimensional regression problems. Linear operator learning has been
studied for diagonal or spectrally structured operators
\cite{de2023convergence}, through an infinite-dimensional regression
formulation as a non-compact inverse problem \cite{mollenhauer2022learning},
in online settings \cite{subedi2024online}, and in kernel/operator estimation
problems \cite{jin2022minimax,zhang2025minimax,bigot2019estimation}. In particular,
\cite{jin2022minimax} studies minimax optimal kernel operator learning for
Hilbert--Schmidt operators between Sobolev reproducing kernel Hilbert spaces,
under capacity, embedding, and source-type assumptions. In contrast, we study
the operator-norm ball \(\mcO(s,s',R)\) under the Sobolev operator norm
\(H^t\to H^{-t'}\). This worst-case loss leads to a two-sided multiscale
geometry in wavelet coordinates. Moreover, our estimator uses finite-resolution
coefficient data and blockwise least squares with scale-adaptive sample usage.
We derive an explicit computational cost for this estimator and show its optimality within the dense least-squares framework. This finite-resolution
computational viewpoint is absent from \cite{jin2022minimax}, whose estimator
depends on full functional data through empirical covariance and
cross-covariance operators.

Another closely related line of work exploits PDE or geometric structure in
the target operator. Examples include learning solution operators for elliptic,
parabolic, and hyperbolic PDEs
\cite{schafer2024sparse,boulle2023learning,boulle2023elliptic,
boulle2022learning,wang2025operator}, pseudo-differential operators
\cite{chen2026convergence}, and geometric integral or scattering operators
\cite{roddenberry2025rates,balaji2026hybrid,su2025learning}. These structured
settings can yield sharp rates by exploiting additional information about the
operator class. Our work is complementary: we do not assume that the operator
comes from a particular PDE, symbol class, kernel geometry, or scattering
model. Instead, we analyze the statistical and computational limits generated
by linearity and Sobolev-scale boundedness. A central feature of the resulting theory is
that the minimax statistical rate and the computational cost needed to attain
it may be governed by different multiscale bottlenecks.

\subsection{Notation}
Throughout the paper, we work on the $d$-dimensional torus
$\mathbb T^d=(\mathbb R/\mathbb Z)^d$ to avoid boundary effects. Similar
formulations on domains with boundary are possible after choosing appropriate
boundary conditions. For $s\in\mathbb R$, let
$H^s(\mathbb T^d)$ denote the usual $L^2$-based Sobolev space on
$\mathbb T^d$. We use $L^2(\mathbb T^d)$ as the pivot space, and identify
the dual of $H^s(\mathbb T^d)$ with $H^{-s}(\mathbb T^d)$ through the
$L^2$ pairing. Thus, if $f\in H^{-s}(\mathbb T^d)$ and
$g\in H^s(\mathbb T^d)$, we write
\[
\langle f,g\rangle
:=
\langle f,g\rangle_{H^{-s}(\mathbb T^d),H^s(\mathbb T^d)}.
\]
This agrees with the usual $L^2$ inner product when $f$ and $g$ belong to $L^2(\mathbb T^d)$. For notational simplicity, we write $H^s$ in place of $H^s(\mathbb T^d)$ throughout the paper, and we omit the subscript in the duality pairing whenever no ambiguity can arise.

For a vector \(x\in\mathbb R^m\), we denote by \(\|x\|_2\) its Euclidean
norm. For a matrix \(B\in\mathbb R^{m_1\times m_2}\), we write \(\|B\|\)
for its spectral norm and \(\|B\|_{\mathrm{F}}\) for its Frobenius norm. More
generally, for a matrix \(B\in \mathbb R^{\Lambda_1\times \Lambda_2}\)
indexed by sets \(\Lambda_1,\Lambda_2\), possibly infinite, we interpret
\(B\) as a linear operator from \(\ell^2(\Lambda_2)\) to
\(\ell^2(\Lambda_1)\), and denote by \(\|B\|\) its operator norm and by
\(\|B\|_{\mathrm{HS}}\) its Hilbert--Schmidt norm. For a symmetric positive semi-definite matrix, we denote by $\sigma_{\min}(B)$ and $\sigma_{\max}(B)$ its smallest and largest eigenvalues, respectively.

For \(a\in\mathbb R\), we write \(a_+:=\max\{a,0\}\), and we use the convention
\(a/0=+\infty\) for \(a>0\). Given two positive sequences $\{a_k\}$ and $\{b_k\}$, we write $a_k \lesssim b_k$ if there exists a constant $c>0$, independent of $k$, such that $a_k \le c\, b_k$ for all $k$. If both $a_k \lesssim b_k$ and $b_k \lesssim a_k$ hold, we write $a_k \asymp b_k$. If the constant $c$ depends on some parameter $\tau$, we write $a_k \lesssim_{\tau} b_k,b_k\lesssim_{\tau} a_k$, and $a_k\asymp_{\tau} b_k$ to indicate this dependence.

\section{Main results}\label{sec:main_results}

For $s,s'\in\mathbb R$, let $\mcL(H^s,H^{-s'})$ denote the space of bounded
linear operators from $H^s$ to $H^{-s'}$, equipped with the operator norm
\[
\|\mcA\|_{H^s\to H^{-s'}}
:=
\sup_{\|w\|_{H^s}\le 1,\ \|v\|_{H^{s'}}\le 1}
|\langle \mcA w,v\rangle|,
\]
where the duality pairing is between $H^{-s'}$ and $H^{s'}$. For $R>0$,
define
\[
\mcO(s,s',R)
:=
\bigl\{
\mcA\in \mcL(H^s,H^{-s'}):
\|\mcA\|_{H^s\to H^{-s'}}\le R
\bigr\}.
\]
Throughout the paper, the radius $R>0$ is fixed, and implicit constants may
depend on $R$ and on the fixed regularity parameters $s$ and $s'$.

We observe input--output data pairs $\{(u_i,f_i)\}_{i=1}^{N}$ generated by
\begin{equation}\label{eq:continuous_model}
f_{i}=\mcA u_{i}+w_{i},\qquad 1\le i\le N,
\end{equation}
where $\mcA\in \mcO(s,s',R)$ is unknown, 
\[
u_{i}\iid \mathcal{GP}(0,(I-\Delta)^{-r_1}),  \qquad  w_i\iid \mathcal{GP}(0,(I-\Delta)^{-r_2}),
\]
for some $r_1,r_2\in \R$, and the input and noise sequences $\{u_i\}$ and $\{w_i\}$ are independent. The Gaussian fields above are understood through the standard spectral construction associated with the Laplacian; in particular, a field with covariance $(I-\Delta)^{-r}$ belongs almost surely to $H^\alpha$ for every $\alpha<r-d/2$ \cite{cox2020regularity,herrmann2020multilevel}. The equality in \eqref{eq:continuous_model} is understood in the distributional sense. To ensure that $\mcA u_i$ is well-defined, we assume $r_1-d/2>s$. Then $u_i\in H^{r_1-d/2-\varepsilon}\hookrightarrow H^s$ almost surely for some $\varepsilon>0$, and hence $\mcA u_i$ is a well-defined element of $H^{-s'}$. 

Given the data $\{(u_i,f_i)\}_{i=1}^N$, the goal is to estimate the unknown operator $\mcA$. We measure estimation error in the Sobolev operator norm $\|\cdot\|_{H^t\to H^{-t'}}$, where $t>s$ and $t'>s'$. The embeddings $H^t\hookrightarrow H^s$ and $H^{-s'}\hookrightarrow H^{-t'}$ ensure that this norm is well-defined on $\mcO(s,s',R)$. Our main theorem gives the minimax risk for this linear operator learning problem.

\begin{theorem}[Minimax risk]\label{thm:main1}
  Let $s,s',r_1,r_2\in\R$ with $r_1-d/2>s$, and let $t>s$ and $t'>s'$. Under the observation model \eqref{eq:continuous_model} with the above input and noise
distributions,
\begin{align*}
\inf_{\widehat{\mcA}}\sup_{\mcA\in \mcO(s,s',R)} \E\, \|\widehat{\mcA}-\mcA\|_{H^{t}\to H^{-t'}} \, \asymp_{\log} \, N^{-\min\big\{\frac{1}{2},\frac{t-s}{2(r_1-s)+d},\frac{t'-s'}{(2(-r_2-s')+d)_{+}}\big\}}.
    \end{align*}
Here the infimum is over all measurable estimators based on $\{(u_i,f_i)\}_{i=1}^N$, and the expectation is taken under the law of the data generated by \eqref{eq:continuous_model} with true operator $\mcA$. The notation $\asymp_{\log}$ denotes equivalence up to powers of $\log N$.
\end{theorem}

\begin{proof}
For the upper bound, Proposition~\ref{prop:upper_bound} establishes a high-probability error bound, at the minimax rate up to logarithmic factors,
for the estimator constructed in Subsection~\ref{subsec:minimax_upper}. Since
this estimator depends on the confidence parameter, Appendix~\ref{app:truncated_estimator}
introduces a stability-truncated version; Lemma~\ref{lemma:truncated_expectation}
shows that the truncated estimator satisfies the corresponding in-expectation
bound. The minimax lower bound is proved in Proposition~\ref{prop:lower_bound}.
Combining the upper and lower bounds proves the theorem.
\end{proof}

\begin{table}[htbp]
\centering
\small
\renewcommand{\arraystretch}{1.5}
\setlength{\tabcolsep}{15pt}
\begin{tabular}{|c|c|}
\hline
\textbf{Parameter} & \textbf{Description} \\
\hline
$N$ & Sample size: number of input--output data pairs $\{(u_i,f_i)\}_{i=1}^N$ \\
\hline
$d$ & Spatial dimension of the torus $\mathbb{T}^d$ \\
\hline
$s, s'$ & Input/output Sobolev index of the operator $\mcA: H^{s}\to H^{-s'}$ \\
\hline
$t,t'$ & Input/output Sobolev index in the error norm $\|\cdot\|_{H^{t}\to H^{-t'}}$ \\
\hline
$r_1,r_2$ & \begin{tabular}[c]{@{}l@{}}
Regularity of the input function/noise: $u\in H^{(r_1-d/2)-}, w\in H^{(r_2-d/2)-}$
\end{tabular} \\
\hline
\end{tabular}
\caption{Summary of key parameters.}
\label{tab:key_parameters}
\end{table}

We make a few remarks on Theorem \ref{thm:main1}.

\begin{remark}[Computational cost]
The upper bound in Theorem~\ref{thm:main1} is achieved by an estimator that
does not require full functional data, but rather only a finite-dimensional
wavelet representation at resolutions chosen according to the sample size; see
Subsection~\ref{subsec:minimax_upper}. Specifically, if the target error is
$\varepsilon$, then the output data are only needed up to scale
$\big\lceil \frac{\log_2(1/\varepsilon)}{t'-s'}  \big\rceil$, while the input data
are used up to scale $\big\lceil
\frac{\log_2(1/\varepsilon)}{\min\{r_1,t\}-s}
\big\rceil$.
Section~\ref{sec:computation} gives a detailed analysis of the resulting
computational cost when the induced finite-dimensional regressions are solved by
dense direct least-squares methods. In particular, Proposition~\ref{prop:cost} shows that the estimator attains the minimax rate
with cost
\[
\Cost(\varepsilon)
\lesssim_{\log}
\varepsilon^{-\max\big\{
2
+
\frac{2d}{\min\{r_1,t\}-s}
+
\frac{(2(r_1-t)+d)_{+}}{t-s},
\ 
\frac{d+(2(-r_2-s')+d)_{+}}{t'-s'}
\big\}},
\]
where the two terms in the maximum correspond to input-side and output-side
computational contributions. Subsection~\ref{subsec:cost_optimality} explains
why these two contributions are unavoidable within the present blockwise dense-regression framework.
\end{remark}

\begin{remark}[Discussion on the convergence rate]
The minimax rate in Theorem~\ref{thm:main1} depends on the parameters
summarized in Table~\ref{tab:key_parameters}. The exponent is the minimum of
three quantities,
\[
\frac12,\qquad
\frac{t-s}{2(r_1-s)+d},\qquad
\frac{t'-s'}{\bigl(2(-r_2-s')+d\bigr)_+}.
\]
The first term is the coarse-scale parametric contribution. The second term is
an input-side contribution: it reflects the difficulty of learning operator
components that act on increasingly oscillatory input directions and produce
low-regularity output information. Since smoother input distributions excite
high-frequency input directions less strongly, increasing \(r_1\) slows down
this part of the rate. The third term is an output-side contribution: it
reflects the difficulty of recovering output components at increasingly low
regularity in the presence of rough observation noise. As \(r_2\) increases,
the noise becomes smoother and this output-side obstruction weakens.

This should be compared with the standard minimax rate in nonparametric
regression \cite{Tsybakov2009,Johnstone2019Gaussian}. In a typical model
\[
y_i=h(x_i)+\xi_i,\qquad \xi_i\iid \mathcal N(0,1),
\]
estimating a real-valued function $h$ on $[0,1]^d$ with Sobolev smoothness \(\beta\) under an $L^2$
loss gives a rate of order
\[
N^{-\frac{\beta}{2\beta+d}}.
\] In the present
problem, the unknown object is a linear operator rather than a function. The
rate therefore has a two-sided structure: one side is associated with how the
operator acts on input directions, and the other with the regularity of the
output components that must be recovered. This two-sided structure gives rise to
the multiscale geometry and computational considerations developed in the following
sections.
\end{remark}

\begin{remark}[Data and noise distributions]
For simplicity of exposition, we take the input and noise distributions to be
\(\mathcal{GP}(0,(I-\Delta)^{-r_1})\) and
\(\mathcal{GP}(0,(I-\Delta)^{-r_2})\), respectively. The analysis only uses the
corresponding multiscale covariance equivalences: in a suitable wavelet basis,
the input and noise covariance operators are diagonally preconditioned by the
Sobolev weights associated with \(r_1\) and \(r_2\); see Lemma \ref{lemma:property_UAW}. Thus the same result holds for more general Gaussian inputs and noises satisfying
analogous elliptic preconditioning bounds, for instance strongly elliptic
pseudo-differential covariance operators; see
\cite{harbrecht2024multilevel,chen2026convergence}.

In Theorem~\ref{thm:main1}, we assume only \(r_1-d/2>s\), which ensures that
\(\mcA u_i\) is well defined, and impose no regularity condition on
\(r_2\). The observation model \eqref{eq:continuous_model} is therefore
understood in the distributional sense: the noise may be a generalized
Gaussian field rather than a function-valued process. In wavelet coordinates,
this gives a heteroscedastic infinite-dimensional regression model with
scale-dependent input and noise variances.
\end{remark}

\section{Main ideas}\label{sec:main_idea}

This section discusses four key ideas underlying the construction and analysis of our estimator developed in Section \ref{sec:convergencerate}. First, we recast the operator learning problem as the problem of learning the wavelet matrix representation of the operator. Second, we propose a local bias-variance decision rule to determine which blocks of this matrix representation are estimated and which are discarded. Third, we show that the same estimation rate can be achieved by more aggressively discarding blocks outside a smaller estimation region. 
Fourth, we describe how to adapt the sample size across scales in order to balance local estimation errors. The third and fourth ideas reduce the computational cost without degrading the estimation error. We describe these ideas in turn in the following subsections. 

\subsection{Wavelet representation and infinite-dimensional regression}\label{ssec:waveletrepresentation}

Here, we recast the operator learning problem as an infinite-dimensional matrix regression problem in wavelet coordinates. This coordinate representation is the starting point for the multiscale bias--variance and computational analysis below.

Fix an orthonormal wavelet basis
\(\{\varphi_{j,k}:j\ge0,\ k\in\nabla_j\}\) of \(L^2(\mathbb T^d)\), where $|\nabla_j|:=\mathrm{Card}(\nabla_j) \asymp 2^{jd}.$
The index \(\lambda=(j,k)\) records the scale \(j\) and location \(k\). We define the wavelet index set 
\[ 
\Lambda:=\{\lambda=(j,k): j\ge0,\ k\in\nabla_j\}. 
\] 
For \(\lambda=(j,k)\), we set \(|\lambda|:=j\) and write \(\varphi_\lambda:=\varphi_{j,k}\).

For a function or distribution $h$, write its wavelet coefficients as
\[
h_\lambda:=\langle h,\varphi_{\lambda}\rangle, \qquad \bfh:=(h_\lambda)_{\lambda\in\Lambda},
\]
where the pairing is understood in the distributional sense when needed. We assume that the wavelets have sufficiently high regularity and sufficiently many vanishing moments so that, for every Sobolev index appearing below, \[ \|h\|_{H^a}^2 \asymp \sum_{\lambda\in\Lambda} 2^{2a|\lambda|} |\langle h,\varphi_\lambda\rangle|^2. 
\] 
For $a\in\R$, define the diagonal operator $\bfD^a$ on $\ell^2(\Lambda)$ by \begin{equation}\label{eq:D} (\bfD^a)_{\lambda,\lambda}=2^{a|\lambda|}, \qquad \lambda\in\Lambda. \end{equation} Then \[ 
\|h\|_{H^a}\asymp \|\bfD^a\bfh\|_{\ell^2}. 
\]

For a linear operator $\mcA$, define its wavelet matrix representation $\mathbf A=(A_{\lambda,\lambda'})_{\lambda,\lambda'\in\Lambda}$ by
\[
\bfA_{\lambda,\lambda'}
:=
\langle \mathcal A \varphi_{\lambda'},\varphi_{\lambda}\rangle.
\]
With this convention, the first index corresponds to the output coordinate and the second to the input coordinate. 
 Formally, $\mcA$ admits the wavelet expansion
\begin{align*}
\mcA=\sum_{\lambda,\lambda'\in \Lambda} \bfA_{\lambda,\lambda'}\, \varphi_{\lambda}\otimes \varphi_{\lambda'}=\sum_{\lambda,\lambda'\in \Lambda} \langle \mcA \varphi_{\lambda'},\varphi_{\lambda} \rangle \varphi_{\lambda}\otimes \varphi_{\lambda'}.
\end{align*} 
By the norm equivalence, for any $\mcA\in\mcO(s,s',R)$ and $t\ge s$, $t'\ge s'$, \begin{equation}\label{eq:matrix_norm} \|\mcA\|_{H^t\to H^{-t'}} \asymp \|\bfD^{-t'}\bfA\bfD^{-t}\|. \end{equation}

In wavelet coordinates, the observation model $f_i=\mcA u_i+w_i$ becomes
\[
\bff_i=\bfA \bfu_i+\bfw_i,\qquad 1\le i\le N.
\] 
Stacking these coefficient vectors row-wise gives \begin{equation}\label{eq:discrete_model} \bfF=\bfU\bfA^\top+\bfW, \end{equation} where \[ \bfF:=[\bff_1,\ldots,\bff_N]^\top\in\R^{N\times\Lambda}, \qquad \bfU:=[\bfu_1,\ldots,\bfu_N]^\top\in\R^{N\times\Lambda}, \qquad \bfW:=[\bfw_1,\ldots,\bfw_N]^\top\in\R^{N\times\Lambda}. \]
Thus, given $\bfU$ and $\bfF$, the operator learning problem reduces to estimating the infinite-dimensional matrix $\bfA$ under the norm $\|\bfD^{-t'}  \cdot  \bfD^{-t}\|$ in \eqref{eq:matrix_norm}.

The following lemma summarizes the structural properties of the wavelet-coordinate model used throughout the analysis.

\begin{lemma}[Properties of $\bfA,\bfU,\bfW$]\label{lemma:property_UAW}   For $\mcA\in\mcO(s,s',R)$, $u_i\iid \mathcal{GP}(0,(I-\Delta)^{-r_1})$, and $w_i\iid \mathcal{GP}(0,(I-\Delta)^{-r_2})$, the corresponding wavelet coefficient matrices $\bfA,\bfU,\bfW$ satisfy the following properties.
    \begin{enumerate}[label=(\Roman*)]
        \item \label{boundedness} \textbf{Boundedness of $\bfA$:} There exists a constant $R'>0$ such that $\| \bfD^{-s'}\bfA \bfD^{-s}\|\le R'$.
        \item \label{diagonal_preconditioning} \textbf{Diagonal preconditioning:}
There exist constants $0<c_{-}<c_{+}< \infty$ such that: 
 \begin{enumerate}[label=(\roman*)]
\item \label{preconditioning_data} 
The population covariance matrix of the input data $\bfC_u:=\E \left[\bfU^{\top}\bfU/N\right]$ satisfies
\[
c_{-}\le \sigma_{\min}(\bfD^{r_1} \bfC_u \bfD^{r_1})\le \sigma_{\max}(\bfD^{r_1} \bfC_u \bfD^{r_1})\le c_{+}.
\]
\item \label{preconditioning_noise} The population covariance matrix of the noise $\bfC_w:=\E \left[\bfW^{\top}\bfW/N\right]$ satisfies 
\[
c_{-}\le \sigma_{\min}(\bfD^{r_2} \bfC_w \bfD^{r_2})\le \sigma_{\max}(\bfD^{r_2} \bfC_w \bfD^{r_2})\le c_{+}.
\]    
\end{enumerate}
\end{enumerate}
\end{lemma}

\begin{proof}
The bound in \ref{boundedness} follows from the assumed
operator bound on $\mcA$ and the norm equivalence \eqref{eq:matrix_norm}. The
preconditioning bounds in \ref{diagonal_preconditioning} are standard
consequences of wavelet diagonal preconditioning for Sobolev-scale
elliptic covariance operators, see
\cite[Proposition~3.2]{chen2026convergence} and
\cite[Propositions~3 and~13]{harbrecht2024multilevel}.
\end{proof}

The operator class $\mcO(s,s',R)$ is represented, up to constants in the wavelet norm equivalence, by the weighted matrix class \[ \mcM(s,s',R') := \left\{ \bfA\in\mathbb R^{\Lambda\times\Lambda}: \|\bfD^{-s'}\bfA\bfD^{-s}\|\le R' \right\}. \] Therefore, under the matrix model \eqref{eq:discrete_model}, we study the minimax risk \[ \inf_{\widehat{\bfA}} \sup_{\bfA\in\mcM(s,s',R')} \E\,\|\bfD^{-t'}(\widehat{\bfA}-\bfA)\bfD^{-t}\|,
\]
where the infimum is over all measurable estimators based on $\bfU$ and $\bfF$.

We summarize the first idea for learning linear operators as follows:
\begin{center}
\fbox{\parbox{0.95\linewidth}{
\centering
\textbf{Idea I.} 
Wavelet coordinates reduce learning linear operators between Sobolev spaces to infinite-dimensional matrix estimation under diagonally weighted operator norms.
}}
\end{center}

\subsection{Bias--variance geometry}

Recall the infinite-dimensional matrix model
\begin{align}\label{eq:discrete_model2}
    \bfF=\bfU\bfA^{\top}+\bfW,
\end{align}
where the unknown matrix \(\bfA\) belongs to \(\mathcal{M}(s,s',R')\). Given \(\bfU\)
and \(\bfF\), our goal is to estimate \(\bfA\) under the weighted operator
norm \(\|\bfD^{-t'}\cdot \bfD^{-t}\|\). In what follows, it is more
convenient to work with \(\bfA^\top\), equipped with the equivalent
transposed loss \(\|\bfD^{-t} \cdot \bfD^{-t'}\|\).

The rows and columns of \(\bfA^\top\) are indexed by the multiscale wavelet
index set \(\Lambda=\{\lambda=(j,k):j\ge0,\ k\in\nabla_j\}\). It is
therefore useful to view \(\bfA^\top\) blockwise. We denote by
\[
    (\bfA^\top)_{j,j'}\in \mathbb R^{|\nabla_j|\times |\nabla_{j'}|}
\]
the \((j,j')\)-th block of \(\bfA^\top\), where
\(|\nabla_j|\asymp 2^{jd}\) and \(|\nabla_{j'}|\asymp 2^{j'd}\).

From the perspective of statistical estimation, it is natural to begin with a
blockwise bias--variance analysis. We describe this heuristic for estimating
\(\bfA^\top\) in the model \eqref{eq:discrete_model2}; rigorous blockwise bias
and variance estimates are given in Appendices~\ref{app:bias}
and~\ref{app:var}, respectively. 
Since the target loss for
\(\bfA^\top\) is \(\|\bfD^{-t} \cdot \bfD^{-t'}\|\), discarding the
\((j,j')\)-th block contributes, after Sobolev weighting, an error of order
\[
    2^{-jt-j't'}\|(\bfA^\top)_{j,j'}\|.
\]
The assumption
\(\|\bfD^{-s'}\bfA\bfD^{-s}\|\le R'\) implies, at the block level,
\[
\|(\bfA^{\top})_{j,j'}\|\lesssim R' 2^{js+j's'}.
\]
Hence the bias caused by discarding the \((j,j')\)-th block is of order
\[
\mathsf{Bias}_{j,j'}\asymp 2^{-j(t-s)-j'(t'-s')}.
\]

We next estimate the corresponding variance scale. The input covariance has
regularity \(r_1\), so the typical size of the design in the \(j\)-th input
block is \(2^{-jr_1}\). The noise covariance has regularity \(r_2\), so the
typical size of the noise in the \(j'\)-th output block is \(2^{-j'r_2}\).
Inverting the design therefore contributes a factor \(2^{jr_1}\), while the
noise contributes a factor \(2^{-j'r_2}\). After converting the resulting block error to the $\|\bfD^{-t} \cdot \bfD^{-t'}\|$ loss
and using Chevet's inequality to control the operator norm of a random matrix
of size \(2^{jd}\times 2^{j'd}\), the square root of the variance contribution
to the blockwise operator-norm error is heuristically
\[
\sqrt{\mathsf{Var}_{j,j'}(N)}
\asymp
\frac{
2^{-jt-j't'+jr_1-j'r_2+\frac d2\max\{j,j'\}}
}{\sqrt N}.
\]

The factor \(2^{\frac d2\max\{j,j'\}}\) is the operator-norm fluctuation
scale of the corresponding random block. Indeed, a Gaussian random matrix of
size \(2^{jd}\times 2^{j'd}\) with independent \(\mathcal{N}(0,1)\) entries has
operator norm of order
\[
2^{jd/2}+2^{j'd/2}
\asymp
2^{\frac d2\max\{j,j'\}}.
\]

Therefore, the local bias--variance balance for deciding whether the
\((j,j')\)-th block should be estimated is
\begin{align}\label{eq:bias-var-tradeoff}
2^{-j(t-s)-j'(t'-s')} \asymp \frac{
2^{-jt-j't'+jr_1-j'r_2+\frac d2\max\{j,j'\}}
}{\sqrt N} .
\end{align}
Equivalently,
\begin{align}\label{eq:line-eq}
j(r_1-s)+j'(-r_2-s')+\frac d2\max\{j,j'\}
=
\frac12\log_2 N .
\end{align}
Notice that this boundary is independent of the loss exponents \(t,t'\).

The boundary \eqref{eq:line-eq} consists of two linear pieces. In the region
\(j\ge j'\), it is given by
\[
j\Bigl(r_1-s+\frac d2\Bigr)+j'(-r_2-s')
=
\frac12\log_2 N,
\]
whereas in the region \(j'\ge j\), it is given by
\[
j(r_1-s)+j'\Bigl(-r_2-s'+\frac d2\Bigr)
=
\frac12\log_2 N.
\]
Thus the boundary passes through the points
{\scriptsize
\[
\left(\frac{\log_2 N}{2(r_1-s)+d},\ 0\right),
\quad
\left(0,\ \frac{\log_2 N}{2(-r_2-s')+d}\right),\quad
\left(
\frac{\log_2 N}{2(r_1-r_2-s-s')+d},\
\frac{\log_2 N}{2(r_1-r_2-s-s')+d}
\right),
\]
}
as illustrated in Figure~\ref{fig:figure1}.

 Recall that the assumption \(u\in H^{(r_1-d/2)_{-}}\hookrightarrow H^s\), which ensures
that \(\mcA u\) is well defined, is equivalent to \(r_1>s+d/2\). In
particular, the coefficient of \(j\) in \eqref{eq:line-eq} is positive. For
simplicity, let us focus on the case \(-r_2-s'>0\). Then the blocks selected by the local bias--variance tradeoff rule
\eqref{eq:bias-var-tradeoff} form the region
\[
j(r_1-s)+j'(-r_2-s')+\frac d2\max\{j,j'\}
\le
\frac12\log_2 N
\]
in the \((j,j')\)-plane.  Since the \((j,j')\)-block has size \(2^{jd}\times 2^{j'd}\), the piecewise
linear cutoff in the scale indices \((j,j')\) corresponds, when viewed in the
full wavelet-index coordinates \(\lambda=(j,k)\) and \(\lambda'=(j',k')\), to
a piecewise hyperbolic cutoff boundary.

\vspace{0.3em}
The resulting convergence rate can be read off from the extremal points of
this selected region. These are
{\scriptsize
\[
(0,0),\quad
\left(\frac{\log_2 N}{2(r_1-s)+d},0\right),
\quad
\left(0,\frac{\log_2 N}{2(-r_2-s')+d}\right), \quad
\left(
\frac{\log_2 N}{2(r_1-r_2-s-s')+d},
\frac{\log_2 N}{2(r_1-r_2-s-s')+d}
\right).
\]
}
\hspace{-0.65em} Evaluating the discarded-block bias $2^{-j(t-s)-j'(t'-s')}$ at these four points yields the four dominant error scales
\[
N^{-1/2},\qquad
N^{-\frac{t-s}{2(r_1-s)+d}},\qquad
N^{-\frac{t'-s'}{2(-r_2-s')+d}},
\qquad
N^{-\frac{t+t'-s-s'}{2(r_1-r_2-s-s')+d}}.
\]
These correspond respectively to the low-frequency block, the input-side
high-frequency boundary, the output-side high-frequency boundary, and the
diagonal transition point \(j=j'\). Therefore the final convergence rate is
given by the slowest of these four rates, namely
\begin{equation}\label{eq:unsimplifiedrate}
   N^{-\min\big\{
\frac{1}{2},\,
\frac{t-s}{2(r_1-s)+d},\,
\frac{t'-s'}{(2(-r_2-s')+d)_{+}},\,
\frac{t+t'-s-s'}{(2(r_1-r_2-s-s')+d)_{+}}
\big\}}. 
\end{equation}
Notice that the diagonal transition point is never the bottleneck. Indeed,
set
\[
a=t-s>0,\qquad b=t'-s'>0,\qquad c=r_1-s>0,\qquad e=-r_2-s'.
\]
Then
\[
\frac{a+b}{(2(c+e)+d)_+}
\ge
\min\left\{
\frac{a}{2c+d},
\frac{b}{(2e+d)_+}
\right\}.
\]
Consequently,
\[
\frac{t+t'-s-s'}{(2(r_1-r_2-s-s')+d)_+}
\ge
\min\left\{
\frac{t-s}{2(r_1-s)+d},
\frac{t'-s'}{(2(-r_2-s')+d)_+}
\right\}.
\]
Thus the diagonal transition point does not affect the final minimum, and the
convergence rate \eqref{eq:unsimplifiedrate} simplifies to
\[
N^{-\min\big\{
\frac{1}{2},\,
\frac{t-s}{2(r_1-s)+d},\,
\frac{t'-s'}{(2(-r_2-s')+d)_+}
\big\}}.
\]

Our first main result, Theorem~\ref{thm:main1}, shows that this heuristic
bias--variance calculation indeed gives the minimax rate, up to logarithmic
factors. A natural estimator attaining this rate is obtained by performing
linear regressions over the blocks selected by the local bias--variance
tradeoff:
\begin{align}\label{eq:bv_tradeoff_line}
j(r_1-s)+j'(-r_2-s')+\frac d2\max\{j,j'\}
\le
\frac12\log_2 N.
\end{align}
We henceforth refer to the set of blocks satisfying
\eqref{eq:bv_tradeoff_line} as the bias--variance region
\(\mathcal{BV}_N\). This is the quadrilateral region in the
\((j,j')\)-plane illustrated in Figure~\ref{fig:figure1}.

 We summarize the second key idea for learning linear operators as follows:
\begin{center}
\fbox{\parbox{0.9\linewidth}{
\centering
\textbf{Idea II.} 
Minimax optimal estimation can be achieved by discarding blocks outside a region determined by a local bias--variance decision rule.
}}
\end{center}

\begin{figure}[htbp]
\centering
% \documentclass[tikz,border=6pt]{standalone}
% \usepackage{amsmath,amssymb}
% \usepackage{pgfplots}
% \usetikzlibrary{arrows.meta}
% \pgfplotsset{compat=1.18}

% \begin{document}

\begin{tikzpicture}
\begin{axis}[
    width=10cm,
    height=10cm,
    xmin=-0.05, xmax=1.2,
    ymin=-0.05, ymax=1.1,
    y dir=reverse,
    axis lines=none,
    xtick=\empty,
    ytick=\empty,
    clip=false,
    samples=300,
]

% key constants
\pgfmathsetmacro{\a}{0.58}
\pgfmathsetmacro{\b}{0.47}
\pgfmathsetmacro{\c}{0.7}
\pgfmathsetmacro{\d}{0.48}

% "axes" j=0 and j'=0 with arrows
\draw[thick,-{Latex[length=3mm,width=2mm]}] (axis cs:0,0) -- (axis cs:1.05,0);
\draw[thick,-{Latex[length=3mm,width=2mm]}] (axis cs:0,0) -- (axis cs:0,0.9);

% axis labels
\node[above right] at (axis cs:1.06,0.02) {$j'$};
\node[below left]  at (axis cs:-0.01,0.89) {$j$};

% origin
\node[above left] at (axis cs:0.01,0.01) {\small{$(0,0)$}};

% two piecewise linear boundaries
\addplot[thick] coordinates {(0,\a) (\b,0)};
\addplot[thick] coordinates {(0,\a) (\d+0.075,\d)};
\addplot[thick] coordinates {(\d+0.075,\d) (\c,0)};

% diagonal dashed line j = j'
% \draw[thick,dashed] (axis cs:0,0) -- (axis cs:0.7,0.7);

% % shade the triangle
% \fill[gray!20]
%     (axis cs:0,0) -- (axis cs:0,\a) -- (axis cs:\b,0) -- cycle;

% % enlarged quadrilateral
% \fill[gray!8]
%     (axis cs:0,0) -- (axis cs:0,\a) -- (axis cs:\d+0.03,\d) -- (axis cs:\c,0) -- cycle;

% minimal triangle
\fill[pattern=north east lines, pattern color=gray!100]
    (axis cs:0,0) -- (axis cs:0,\a) -- (axis cs:\b,0) -- cycle;

\node[left] at (axis cs:\b+0.17,-0.055)
{\small{$\big(0,\frac{ (t-s)\log_2 N}{(t'-s')(2(r_1-s)+d)}\big)$}};

\node[left] at (axis cs:0,\a)
{\small{$\big(\frac{\log_2 N}{2(r_1-s)+d},0\big)$}};

\node[above] at (axis cs:\c+0.12,0)
{\small{$\big(0,\frac{\log_2 N}{2(-r_2-s')+d}\big)$}};

\node at (axis cs:0.75,0.56) {\small{$\big(\frac{\log_2 N}{2(r_1-r_2-s-s')+d},\frac{\log_2 N}{2(r_1-r_2-s-s')+d}\big)$}};

% region labels
\node at (axis cs:0.18,0.19) {\small{$\mathcal{E}_N$}};
\node at (axis cs:0.39,0.35) {\small{$\mathcal{BV}_N \backslash \,\mathcal{E}_N$}};

\end{axis}
\end{tikzpicture}

% \end{document}

% % region labels
% \node at (axis cs:0.18,0.15) {\small{$\mathcal{E}_N$}};
% \node at (axis cs:0.15,0.39) {\small{$\mathcal{R}_N \backslash \,\mathcal{E}_N$}};
\caption{The bias--variance region \(\mathcal{BV}_N\) (quadrilateral) and the
minimal estimation region \(\mcE_N\) (shaded triangle) in the $(j,j')$-plane, when
\(-r_2-s'>0\) and \eqref{eq:rate_dominate} holds. }
\label{fig:figure1}
\end{figure}
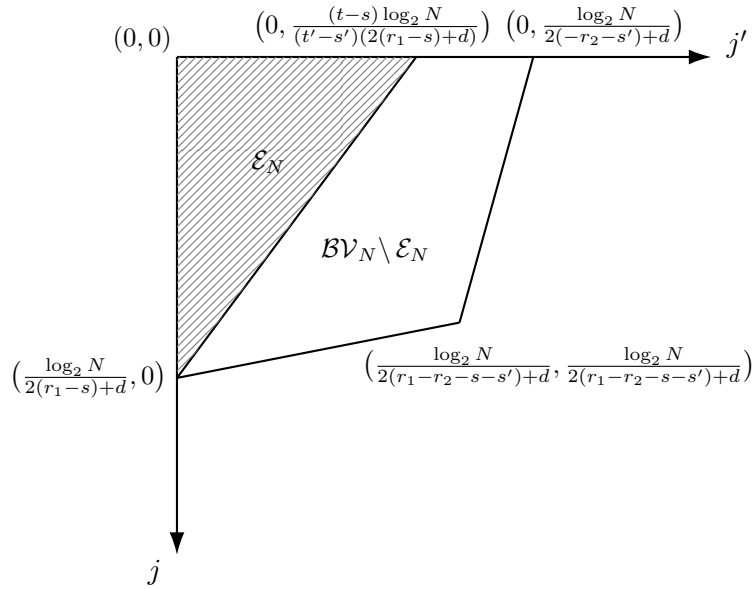

\subsection{Minimal estimation region}\label{subsec:minimal_estimation_region}
From a computational perspective, one may ask whether it is necessary to estimate every block in the bias-variance region \(\mathcal{BV}_N\). 
It turns out that the same minimax rate can be achieved using a smaller, in fact minimal, estimation region, thereby reducing the computational cost. To illustrate the idea, suppose for simplicity that the minimax rate is
dominated by the $j$-side boundary, namely
\begin{align}\label{eq:rate_dominate}
N^{-\min\big\{
\frac{1}{2},\,
\frac{t-s}{2(r_1-s)+d},\,
\frac{t'-s'}{(2(-r_2-s')+d)_+}
\big\}}
=
N^{-\frac{t-s}{2(r_1-s)+d}} .
\end{align}
Equivalently, the rate is determined by the block near
\[
\left(\frac{\log_2 N}{2(r_1-s)+d},0\right).
\]

In this case, it is not necessary to estimate all blocks satisfying
\eqref{eq:bv_tradeoff_line}. To preserve the rate
\(N^{-\frac{t-s}{2(r_1-s)+d}},\) 
we can discard additional blocks if the bias incurred by doing so is below this target accuracy. 
Thus we may discard all
blocks satisfying
\[
2^{-j(t-s)-j'(t'-s')}
\le
N^{-\frac{t-s}{2(r_1-s)+d}}.
\]
Equivalently, it suffices to estimate only those blocks satisfying
\begin{align}\label{eq:estimation_line}
j(t-s)+j'(t'-s')
\le
\frac{t-s}{2(r_1-s)+d}\log_2 N .
\end{align}
The corresponding estimation region is the shaded triangle in
Figure~\ref{fig:figure1}, with vertices
\[
(0,0),\qquad
\left(\frac{\log_2 N}{2(r_1-s)+d},0\right),
\qquad
\left(
0,
\frac{(t-s)\log_2 N}{(t'-s')(2(r_1-s)+d)}
\right).
\] By the assumption
\eqref{eq:rate_dominate}, we have
\[
\frac{(t-s)\log_2 N}{(t'-s')(2(r_1-s)+d)}
<
\frac{\log_2 N}{2(-r_2-s')+d}.
\]
Therefore, the region \eqref{eq:estimation_line} is strictly contained in the
bias--variance region \eqref{eq:bv_tradeoff_line}. Hence estimating every
block selected by the local bias--variance rule is not computationally
 optimal. To achieve the minimax rate
\(N^{-\frac{t-s}{2(r_1-s)+d}}\), it is enough to estimate the smaller region
\eqref{eq:estimation_line}; this region is the minimal possible estimation
region at the level of this blockwise bias criterion. This shows that the
local bias--variance rule, while statistically natural, does not by itself
lead to a computationally optimal estimator.

We summarize the third key idea for learning linear operators as follows:
\begin{center}
\fbox{\parbox{0.9\linewidth}{
\centering
\textbf{Idea III.} The 
local bias--variance rule does not yield the minimal minimax-optimal region, and hence does not lead to computationally optimal estimation.
}}
\end{center}

\subsection{Scale-adaptive sample usage}\label{subsec:sample_usage}

We have now identified the minimal estimation region. Thus, to achieve the
minimax rate in \eqref{eq:rate_dominate}
\[
N^{-\frac{t-s}{2(r_1-s)+d}}=:N^{-\gamma},
\]
it suffices to estimate all blocks whose scale indices satisfy
\eqref{eq:estimation_line}.

Starting from the matrix model
\[
\bfF=\bfU\bfA^{\top}+\bfW,
\]
we observe that it can be written block-columnwise. For any \(j'\ge 0\), the
\(j'\)-th column block satisfies
\[
\bfF_{j'}=\bfU(\bfA^{\top})_{\cdot,j'}+\bfW_{j'}.
\]
Here
\[
\bfF_{j'}\in\mathbb R^{N\times |\nabla_{j'}|},
\qquad
(\bfA^{\top})_{\cdot,j'}\in \mathbb R^{\Lambda\times |\nabla_{j'}|},
\qquad
\bfW_{j'}\in\mathbb R^{N\times |\nabla_{j'}|}
\]
denote the \(j'\)-th column blocks of \(\bfF\), \(\bfA^\top\), and \(\bfW\),
respectively.

A first natural estimator is to use all \(N\) samples and perform least
squares regression for all blocks in the minimal estimation region
\eqref{eq:estimation_line}. However, this is still not computationally
optimal. The reason is that the local statistical difficulty is not uniform
throughout the region.

Indeed, under the simplifying assumption \eqref{eq:rate_dominate}, the
minimax rate is determined by the input-side boundary point
\[
\left(\frac{\log_2 N}{2(r_1-s)+d},0\right),
\]
and equals \(N^{-\gamma}=N^{-\frac{t-s}{2(r_1-s)+d}}\). Now consider instead
the output-side endpoint of the minimal estimation region,
\[
\left(
0,
\frac{(t-s)\log_2 N}{(t'-s')(2(r_1-s)+d)}
\right).
\]
This point lies on the boundary of \eqref{eq:estimation_line}, so its
discarded-block bias is exactly of the target order \(N^{-\gamma}\). However,
by \eqref{eq:rate_dominate}, this point lies strictly inside the larger
bias--variance region \eqref{eq:bv_tradeoff_line}. Therefore, at this block,
the square root of the variance using all \(N\) samples is strictly smaller than its bias:
\[
\sqrt{\mathsf{Var}_{0,\frac{(t-s)\log_2 N}{(t'-s')(2(r_1-s)+d)}} (N)}
\ll
\mathsf{Bias}_{0,\frac{(t-s)\log_2 N}{(t'-s')(2(r_1-s)+d)}} .
\]
Since the block lies on the boundary \eqref{eq:estimation_line}, the bias is
of order \(N^{-\gamma}\). Hence
\[
\sqrt{\mathsf{Var}_{0,\frac{(t-s)\log_2 N}{(t'-s')(2(r_1-s)+d)}} (N)}
\ll
\mathsf{Bias}_{0,\frac{(t-s)\log_2 N}{(t'-s')(2(r_1-s)+d)}}
\asymp
N^{-\gamma}.
\]
Thus, if all \(N\) samples are used, the estimation error for this block is
much smaller than the global target rate. In other words, the block near
\[
\left(\frac{\log_2 N}{2(r_1-s)+d},0\right)
\]
is statistically the hardest one, which contributes to the minimax rate $N^{-\frac{t-s}{2(r_1-s)+d}}$, whereas the block near
\[
\left(
0,
\frac{(t-s)\log_2 N}{(t'-s')(2(r_1-s)+d)}
\right)
\]
is easier to estimate.

This observation creates room for computational improvement. Since the square root of the
variance at this easier block is already much smaller than \(N^{-\gamma}\),
we do not need to use all \(N\) samples to estimate it. Instead, for each
block \((j,j')\), one may in principle use only \(N_{j,j'}\le N\) samples,
chosen so that the variance of the \((j,j')\)-block is raised to the global
target level \(N^{-\gamma}\). In other words, the guiding principle is
\[
\sqrt{\mathsf{Var}_{j,j'}(N_{j,j'})}\asymp N^{-\gamma},
\]
whenever the block \((j,j')\) belongs to the minimal estimation region \eqref{eq:estimation_line}.
Thus the number of samples used for each block should depend on its local
statistical difficulty, rather than being fixed uniformly at the full sample
size \(N\) throughout the entire region.

In the actual estimator constructed in Subsection \ref{subsec:minimax_upper}, we use an even simpler
column-adaptive allocation. Namely, for each column scale \(j'\), we choose a
sample size \(N_{j'}\) depending only on \(j'\), and use these \(N_{j'}\)
samples to estimate all relevant blocks in the \(j'\)-th column. This
columnwise allocation already achieves the optimal computational cost; a
fully blockwise allocation \(N_{j,j'}\) would not improve the leading-order
complexity.

We finally summarize the fourth key idea for learning linear operators as follows:
\begin{center}
\fbox{\parbox{0.9\linewidth}{
\centering
\textbf{Idea IV.} Region-adaptive sample sizes balance local estimation errors,
preserve the global minimax rate, and achieve optimal computational complexity.
}}
\end{center}

\section{Minimax rate}\label{sec:convergencerate}

\subsection{Minimax upper bound}\label{subsec:minimax_upper}

We first set up some matrix notation. Recall that
\[
\Lambda=\{\lambda=(j,k):j\ge0,\ k\in\nabla_j\},
\qquad
|\nabla_j|\asymp 2^{jd},
\]
where \(\nabla_j\) is the set of location indices at scale \(j\). For a set of
scales \(I\subseteq\mathbb{N}_0\), define
\[
\Lambda_I:=\{(j,k):j\in I,\ k\in\nabla_j\}.
\]
In particular, we write
\[
\Lambda_j:=\Lambda_{\{j\}}=\{j\}\times\nabla_j.
\]
For a matrix \(\bfB\) indexed by \(\Lambda\times\Lambda\), we use the
scale-level notation
\[
\bfB_{I,K}:=\bfB_{\Lambda_I,\Lambda_K}.
\]
In particular,
\[
\bfB_{I,j}:=\bfB_{\Lambda_I,\Lambda_j},
\qquad
\bfB_{j,I}:=\bfB_{\Lambda_j,\Lambda_I}.
\]
We use ``\(\cdot\)'' to indicate that no restriction is imposed in that index;
for instance,
\[
\bfB_{\cdot,j}:=\bfB_{\Lambda,\Lambda_j},
\qquad
\bfB_{I,\cdot}:=\bfB_{\Lambda_I,\Lambda}.
\]
For the data and noise matrices, whose rows are indexed by samples and whose
columns are indexed by wavelets, we write
\[
\bfU_I:=\bfU_{\cdot,\Lambda_I},
\qquad
\bfF_{j'}:=\bfF_{\cdot,\Lambda_{j'}},
\qquad
\bfW_{j'}:=\bfW_{\cdot,\Lambda_{j'}}.
\]
Thus \(\bfU_I\in\mathbb R^{N\times |\Lambda_I|}\), while
\(\bfF_{j'},\bfW_{j'}\in\mathbb R^{N\times |\nabla_{j'}|}\).

We define the exponent
\begin{align}\label{eq:gamma}
\gamma:=\min\left\{\frac{1}{2},\,
\frac{t-s}{2(r_1-s)+d},\,
\frac{t'-s'}{\bigl(2(-r_2-s')+d\bigr)_{+}}\right\}.
\end{align}
As stated in Theorem~\ref{thm:main1} and proved below, \(N^{-\gamma}\), up to logarithmic factors, is the minimax rate. The estimator below is constructed to attain this rate.

The discussion in Subsection~\ref{subsec:minimal_estimation_region} motivates a minimal estimation region: it suffices to estimate those coefficients whose omission may contribute bias larger than the target accuracy $N^{-\gamma}$. This region, formalized in the following definition, specifies the entries of $\bfA^\top$ that will be estimated.

\begin{definition}[Estimation region]\label{def:estimation_region}
For each sample size \(N\), the estimation region is defined as the finite subset of \(\Lambda \times \Lambda\)
\[
\mcE_N := \Big\{(\lambda,\lambda')\in \Lambda\times\Lambda:\,
j(t-s)+j'(t'-s')\le \gamma \log_2 N \Big\}.
\]
\end{definition}
With a slight abuse of notation, we also write \((j,j')\in \mcE_N\) to mean that
the corresponding scales satisfy $j(t-s)+j'(t'-s')\le \gamma\log_2 N.$

We next write \eqref{eq:discrete_model} block-columnwise. 
For any \(j'\ge 0\),
\begin{align}\label{eq:model_column}
\bfF_{j'}=\bfU(\bfA^{\top})_{\cdot,j'}+\bfW_{j'}.
\end{align}
Thus each column scale \(j'\) gives an \emph{infinite-dimensional multivariate regression} problem, with response \(\bfF_{j'}\), design \(\bfU\), coefficient matrix \((\bfA^\top)_{\cdot,j'}\), and noise \(\bfW_{j'}\). Our goal is to estimate the entries of \(\bfA^\top\) in \(\mcE_N\) by solving these regressions columnwise.

However, directly regressing only on the target region \(\mcE_N\) can induce a non-negligible \emph{omitted-variable bias}: the components outside \(\mcE_N\) may still interact with the retained components through the empirical design; see Remark \ref{remark:nested-regression} for a detailed explanation. To control this bias, we regress over a slightly enlarged region and then restrict the resulting estimator back to \(\mcE_N\). This leads to the following regression region:

\begin{definition}[Regression region]\label{def:regression_region}
For each sample size \(N\), the regression region is defined as the finite subset of \(\Lambda \times \Lambda\)
\[
\mcR_N := \Big\{(\lambda,\lambda')\in \Lambda\times\Lambda:\,
j(\min\{r_1,t\}-s)+j'(t'-s')\le \gamma \log_2 N \Big\}.
\]
\end{definition}

\begin{figure}[htbp]
\centering
% \documentclass[tikz,border=6pt]{standalone}
% \usepackage{amsmath,amssymb}
% \usepackage{pgfplots}
% \usetikzlibrary{arrows.meta}
% \pgfplotsset{compat=1.18}

% \begin{document}

\begin{tikzpicture}
\begin{axis}[
    width=10cm,
    height=10cm,
    xmin=-0.05, xmax=1.2,
    ymin=-0.05, ymax=1.1,
    y dir=reverse,
    axis lines=none,
    xtick=\empty,
    ytick=\empty,
    clip=false,
    samples=300,
]

% key constants
\pgfmathsetmacro{\a}{0.6}
\pgfmathsetmacro{\b}{0.37}
\pgfmathsetmacro{\c}{0.8}

% "axes" j=0 and j'=0 with arrows
\draw[thick,-{Latex[length=3mm,width=2mm]}] (axis cs:0,0) -- (axis cs:1.05,0);
\draw[thick,-{Latex[length=3mm,width=2mm]}] (axis cs:0,0) -- (axis cs:0,0.9);

% axis labels
\node[above right] at (axis cs:1.06,0.02) {$j'$};
\node[below left]  at (axis cs:-0.01,0.89) {$j$};

% origin
\node[above left] at (axis cs:0.01,0.01) {\small{$(0,0)$}};

% two piecewise linear boundaries
\addplot[thick] coordinates {(0,\b) (\c,0)};
\addplot[thick] coordinates {(0,\a) (\c,0)};

\fill[pattern=north east lines, pattern color=gray!100]
    (axis cs:0,0) -- (axis cs:0,\b) -- (axis cs:\c,0) -- cycle;

% labels of special points

\node[left] at (axis cs:0,\b)
{\small{$\big(\frac{\gamma \log_2 N}{t-s},\,0\big)$}};

\node[left] at (axis cs:0,\a)
{\small{$\big(\frac{\gamma \log_2 N}{\min\{r_1,t\}-s},\,0\big)$}};

\node[above] at (axis cs:\c,0)
{\small{$\big(0,\,\frac{\gamma \log_2 N}{t'-s'}\big)$}};

% region labels
\node at (axis cs:0.18,0.15) {\small{$\mathcal{E}_N$}};
\node at (axis cs:0.15,0.39) {\small{$\mathcal{R}_N \backslash \,\mathcal{E}_N$}};

\end{axis}
\end{tikzpicture}

% \end{document}
\caption{The estimation region \(\mcE_N\) (shaded triangle) and the regression region \(\mcR_N\) (triangle) in the \((j,j')\)-plane.}
\label{fig:figure2}
\end{figure}
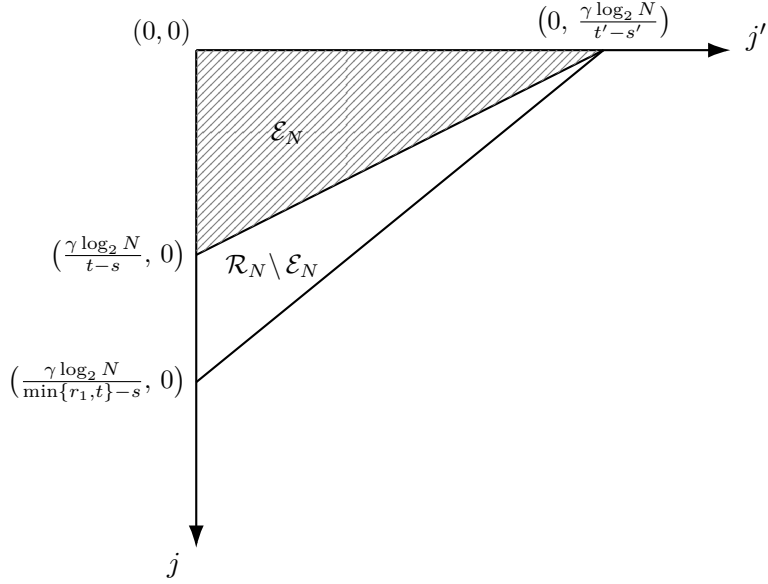

Figure \ref{fig:figure2} illustrates the estimation region $\mcE_N$ and the regression region $\mcR_N$ in the $(j,j')$-plane.  When \(r_1 \ge t\), these two regions coincide; when \(r_1<t\), the regression region strictly contains the estimation region, \(\mcE_N\subsetneq \mcR_N\).

To define the estimator, we introduce the columnwise cutoffs associated with these two regions. For each relevant column scale \(j'\ge 0\), define
 \begin{align}\label{eq:j_point_estimation}
J(j'):=\left\lceil \frac{\gamma \log_2 N-j'(t'-s')}{t-s}\right\rceil,\qquad I_{j'}:=\left\{0,1,\ldots, J(j')\right\}, \ \ \ I^c_{j'}:= \mathbb{N}_0 \setminus I_{j'},
\end{align}
where \(J(j')\) is the largest \(j\)-level retained by \(\mcE_N\) at the fixed column scale \(j'\). Similarly,
\begin{align}\label{eq:j_point_regression}
\widetilde{J}(j'):=\left\lceil \frac{\gamma \log_2 N-j'(t'-s')}{\min\{r_1,t\}-s}\right\rceil,\qquad \widetilde{I}_{j'}:=\left\{0,1,\ldots, \widetilde{J}(j')\right\}, \  \ \ \widetilde{I}^c_{j'}:= \mathbb{N}_0 \setminus \widetilde{I}_{j'}
\end{align}
where $\widetilde{J}(j')$ is the largest \(j\)-level retained by \(\mcR_N\) at the fixed column scale \(j'\).

Our estimator is constructed in two steps. We first perform regressions over the enlarged regression region \(\mcR_N\), and then restrict the resulting estimate to the estimation region \(\mcE_N\). As discussed in Subsection~\ref{subsec:sample_usage}, the regression is carried out with column-adaptive sample sizes: for the \(j'\)-th column block, we use \(N_{j'}\) samples, where \(N_{j'}\) depends on \(j'\), rather than the full sample size \(N\). We will show that this column-adaptive choice achieves the minimax rate while reducing the computational cost to the optimal order.

The complete definitions of $\widehat{\bfA}$ and $\widehat{\mcA}$ are given below. 

\begin{tcolorbox}[scale =1,
  enhanced,
  colback=white,
  frame hidden, boxrule=0pt,
  borderline north={0.8pt}{0pt}{black},
  borderline north={0.8pt}{17pt}{black!60!white},
  borderline south={0.8pt}{0pt}{black!60!white},
  title=\textbf{Algorithm: Construction of the Estimator \(\widehat{\mcA}\)},
  label=alg:construction,
  colbacktitle=white,
  coltitle=black,
  titlerule=0pt,
  left=1mm,right=1mm,top=1mm,bottom=1mm
]

\noindent\textbf{Input:} Confidence level $\delta>0$, input data matrix \(\bfU_{\widetilde{I}_{0}}\), output data blocks \(\bfF_{j'}\) for $j'\le \big\lceil \frac{\gamma\log_2 N}{t'-s'} \big\rceil$.

\noindent\textbf{Output:} Estimator \(\widehat{\mcA}\).

\vspace{0.3em}
\noindent\textbf{Step 1:} Columnwise regression.
\vspace{0.3em}

For each $j'=0,1,\dots,\big\lceil \frac{\gamma\log_2 N}{t'-s'} \big\rceil$, choose \(N_{j'}\) satisfying
\begin{align}\label{eq:Njprime_box}
N_{j'}
\asymp
\max\left\{2^{\widetilde J(j')d}+\log((\log N)/\delta), \
N^{2\gamma}
\sup_{0\le j\le J(j')}
2^{-2jt-2j't'+2jr_1-2j'r_2+d\max\{j,j'\}}\right\}.
\end{align}
Let \(\bfU^{(j')}_{\widetilde I_{j'}}\) and \(\bfF^{(j')}_{j'}\) denote the
submatrices of \(\bfU_{\widetilde I_{j'}}\) and \(\bfF_{j'}\) formed by
selecting the same \(N_{j'}\) sample rows. Define the \(j'\)-th column block of \(\widehat{\bfA}^{\top}\):
\begin{align}\label{eq:estimator_column_box}
(\widehat{\bfA}^{\top})_{\widetilde I_{j'},j'}
:=
\big((\bfU^{(j')}_{\widetilde I_{j'}})^{\top}
\bfU^{(j')}_{\widetilde I_{j'}}\big)^{-1}
(\bfU^{(j')}_{\widetilde I_{j'}})^{\top}
\bfF^{(j')}_{j'} .
\end{align}
All other entries of \(\widehat{\bfA}^{\top}\) are set to zero.

\vspace{0.3em}
\noindent\textbf{Step 2:} Final restriction and operator construction.
\vspace{0.3em}

Define the operator estimator \(\widehat{\mcA}\) by retaining only the entries of \(\widehat{\bfA}^{\top}\) indexed by \(\mcE_N\):
\begin{align}\label{eq:estimator_operator_box}
\widehat{\mcA}
:= \sum_{(\lambda,\lambda') \in \mcE_N}
(\widehat{\bfA}^{\top})_{\lambda,\lambda'}\,
\varphi_{\lambda'} \otimes \varphi_{\lambda}.
\end{align}
\end{tcolorbox}

\smallskip

\begin{remark}[Choice of $N_{j'}$]

The two terms in the definition of \(N_{j'}\) play different roles. The term
\(2^{\widetilde J(j')d}+\log((\log N)/\delta)\) ensures that the columnwise least-squares problem over the enlarged regression support is well-conditioned uniformly over the relevant column scales. The logarithmic factor comes from taking a union bound over \(\mcO(\log N)\) possible values of \(j'\).

The second term controls the stochastic error: for every retained block \((j,j')\in\mcE_N\), it ensures that the weighted blockwise standard deviation is at most of order \(N^{-\gamma}\). Since the exponent inside the supremum is a convex piecewise-affine function of \(j\), the supremum is attained at one of the endpoints \(j=0\) or \(j=J(j')\).

Finally, the choice is feasible: the definition of \(\gamma\) implies that the right-hand side of \eqref{eq:Njprime_box} is \(\lesssim N\) uniformly over all relevant \(j'\). Thus, after adjusting constants, the estimator uses at most the available sample size; in the borderline case, it simply uses all \(N\) samples.
\end{remark}

\smallskip

The following proposition shows that the estimator constructed above attains the minimax rate $N^{-\gamma}$, up to logarithmic factors, with high probability. An in-expectation bound is established in
Appendix~\ref{app:truncated_estimator} for a stability-truncated version of the same estimator.

\begin{proposition}[Minimax upper bound]\label{prop:upper_bound}

Under the setting of Theorem~\ref{thm:main1}, the following holds. For any $\delta\in (0,1)$, if $N\gtrsim  \log(1/\delta)$, then with probability at least $1-\delta$, the estimator $\widehat{\mcA}$ defined in \eqref{eq:estimator_operator_box} satisfies
\[
\|\widehat{\mcA}-\mcA\|_{H^t\to H^{-t'}}
\lesssim N^{-\gamma}\,
\sqrt{\log\!\Big(\frac{N}{\delta}\Big)}\,
\log N.
\]
\end{proposition}

\begin{proof}

For a matrix \(\bfB\in \mathbb R^{\Lambda\times\Lambda}\) and an index set
\(S\subseteq \Lambda\times\Lambda\), we denote by \(\bfB_S\) the matrix obtained
by retaining the entries of \(\bfB\) indexed by \(S\) and setting all remaining
entries to zero. We write \(\bfB_{S^c}\) for the analogous restriction to
\(S^c:=(\Lambda\times\Lambda)\setminus S\). Thus,
\[
\bfB=\bfB_S+\bfB_{S^c}.
\]

We begin with the operator-norm equivalence
\begin{align}\label{eq:error_decomposition_aux1}
\|\widehat{\mcA}-\mcA\|_{H^{t}\to H^{-t'}}
&\asymp
\| \bfD^{-t}((\widehat{\bfA}^{\top})_{\mcE_N}-\bfA^{\top})\bfD^{-t'}\| \nonumber \\
&\le
\| \bfD^{-t} (\bfA^{\top})_{\mcE_N^c} \bfD^{-t'} \|
+
\|\bfD^{-t}((\widehat{\bfA}^{\top})_{\mcE_N}-(\bfA^{\top})_{\mcE_N})\bfD^{-t'}\| \nonumber\\
&=
\underbrace{\| \bfD^{-t} (\bfA^{\top})_{\mcE_N^c} \bfD^{-t'} \|}_{=:\ErrorBias}
+
\underbrace{\|\bfD^{-t}(\widehat{\bfA}^{\top}-\bfA^{\top})_{\mcE_N}\bfD^{-t'}\|}_{=:\ErrorEstimation}.
\end{align}

Recall that, for each column scale \(j'\le \lceil \frac{\gamma \log_2 N}{t'-s'} \rceil \), the estimator in
\eqref{eq:estimator_column_box} is defined by
\[
(\widehat{\bfA}^{\top})_{\widetilde I_{j'},j'}
:=
\big((\bfU^{(j')}_{\widetilde I_{j'}})^{\top}
\bfU^{(j')}_{\widetilde I_{j'}}\big)^{-1}
(\bfU^{(j')}_{\widetilde I_{j'}})^{\top}
\bfF^{(j')}_{j'} .
\]

Moreover, the \(j'\)-th block-column equation in \eqref{eq:model_column},
restricted to the corresponding sample rows, can be written as
\[
\bfF^{(j')}_{j'}
=
\bfU^{(j')}(\bfA^{\top})_{\cdot,j'}+\bfW^{(j')}_{j'}
=
\bfU^{(j')}_{\widetilde I_{j'}}(\bfA^\top)_{\widetilde I_{j'},j'}
+
\bfU^{(j')}_{\widetilde I_{j'}^c}(\bfA^\top)_{\widetilde I_{j'}^c,j'}
+
\bfW^{(j')}_{j'} .
\]
Here \((\cdot)^{(j')}\) denotes restriction to the same set of
$\min\{N_{j'},N\}$ sample rows. Substituting this block-column equation into the definition of the estimator yields
\[
(\widehat{\bfA}^{\top}-\bfA^\top)_{\widetilde I_{j'},j'}
=
\underbrace{
\big((\bfU^{(j')}_{\widetilde I_{j'}})^{\top}
\bfU^{(j')}_{\widetilde I_{j'}}\big)^{-1}
(\bfU^{(j')}_{\widetilde I_{j'}})^{\top}
\bfU^{(j')}_{\widetilde I_{j'}^c}
(\bfA^\top)_{\widetilde I_{j'}^c,j'}
}_{=:\mathsf{OVB}_{j'}}
+
\underbrace{
\big((\bfU^{(j')}_{\widetilde I_{j'}})^{\top}
\bfU^{(j')}_{\widetilde I_{j'}}\big)^{-1}
(\bfU^{(j')}_{\widetilde I_{j'}})^{\top}
\bfW^{(j')}_{j'}
}_{=:\mathsf{Var}_{j'}}.
\]

Let \(\bfB^{\mathrm{ovb}}\) and \(\bfB^{\mathrm{var}}\) be the matrices defined block-columnwise by
\[
(\bfB^{\mathrm{ovb}})_{\widetilde I_{j'},j'}:=\mathsf{OVB}_{j'},
\qquad
(\bfB^{\mathrm{var}})_{\widetilde I_{j'},j'}:=\mathsf{Var}_{j'},
\]
for all relevant \(j'\), with all remaining entries set to zero. Therefore,
\[
(\widehat{\bfA}^{\top}-\bfA^\top)_{\mcE_N}
=
(\bfB^{\mathrm{ovb}})_{\mcE_N}
+
(\bfB^{\mathrm{var}})_{\mcE_N}.
\]
Hence,
\[
\ErrorEstimation
\le
\underbrace{
\|\bfD^{-t}(\bfB^{\mathrm{ovb}})_{\mcE_N}\bfD^{-t'}\|
}_{=:\ErrorOVB}
+
\underbrace{
\|\bfD^{-t}(\bfB^{\mathrm{var}})_{\mcE_N}\bfD^{-t'}\|
}_{=:\ErrorVar}.
\]

Consequently,
\[
\|\widehat{\mcA}-\mcA\|_{H^{t}\to H^{-t'} }
\lesssim  \ErrorBias+\ErrorOVB+\ErrorVar.
\]
Lemma \ref{lemma:Error-Bias} gives the deterministic bound
\[
\ErrorBias
=
\| \bfD^{-t} (\bfA^{\top})_{\mcE_N^c} \bfD^{-t'} \|
\lesssim R' N^{-\gamma}.
\]
Moreover, Lemmas \ref{lemma:Error-OVB} and \ref{lemma:Error-Var} show that, for any \(\delta\in(0,1)\), if
\begin{align}\label{eq:sample_size_condition}
N\gtrsim \log (1/\delta),
\end{align}
then with probability at least \(1-\delta\),
\[
\ErrorOVB =\|\bfD^{-t}(\bfB^{\mathrm{ovb}})_{\mcE_N}\bfD^{-t'}\| \lesssim N^{-\gamma} \log N,
\]
and
\[
\ErrorVar = \|\bfD^{-t}(\bfB^{\mathrm{var}})_{\mcE_N}\bfD^{-t'}\|\lesssim
N^{-\gamma}
\sqrt{\log\!\Big(\frac{N}{\delta}\Big)}
\log N.
\]
Combining these estimates yields
\[
\|\widehat{\mcA}-\mcA\|_{H^{t}\to H^{-t'}}
\lesssim
R'N^{-\gamma}
+
N^{-\gamma}\log N
+
N^{-\gamma}\sqrt{\log\!\Big(\frac{N}{\delta}\Big)}\log N,
\]
with probability at least \(1-\delta\). This proves the claim.
\end{proof}

In the proof of Proposition \ref{prop:upper_bound}, we invoked Lemmas~\ref{lemma:Error-Bias}, \ref{lemma:Error-OVB}, and \ref{lemma:Error-Var} to bound, respectively, the bias, the omitted-variable bias, and the variance term. The proofs of these results are provided in Appendices~\ref{app:bias}, \ref{app:ovb}, and~\ref{app:var}. Recall that $\gamma$ is the exponent defined in \eqref{eq:gamma}.

\begin{lemma}[Bound for $\ErrorBias$]
\label{lemma:Error-Bias}
Under the setting of Theorem~\ref{thm:main1}, it holds that
\[
\ErrorBias=\| \bfD^{-t} (\bfA^{\top})_{\mcE_N^c} \bfD^{-t'} \|\lesssim N^{-\gamma}.
\]
\end{lemma}

\begin{lemma}[Bound for $\ErrorOVB$]\label{lemma:Error-OVB}
Under the setting of Theorem~\ref{thm:main1}, the following holds. For any $\delta\in (0,1)$, if $N\gtrsim \log (1/\delta)$, then with probability at least $1-\delta$,
\begin{align*}
\ErrorOVB=\|\bfD^{-t}(\bfB^{\mathrm{ovb}})_{\mcE_N}\bfD^{-t'}\|\lesssim  N^{-\gamma}\log N.
\end{align*}
\end{lemma}

\begin{lemma}[Bound for $\ErrorVar$]\label{lemma:Error-Var}
Under the setting of Theorem~\ref{thm:main1}, the following holds. For any $\delta\in (0,1)$, if $N\gtrsim\log (1/\delta)$, then with probability at least $1-\delta$,
\[
\ErrorVar=\|\bfD^{-t}(\bfB^{\mathrm{var}})_{\mcE_N}\bfD^{-t'}\|\lesssim N^{-\gamma}\,
\sqrt{\log\!\Big(\frac{N}{\delta}\Big)}\,
\log N.
\]
\end{lemma}

\smallskip

\begin{remark}[Direct regression versus nested-support regression]\label{remark:nested-regression}
We briefly explain, through a finite-dimensional toy model, why the
nested-support regression step reduces the omitted-variable bias without
changing the variance order.

Fix a column scale \(j'\). The \(j'\)-th block-column of the matrix model is
\[
\bfF_{j'}=\bfU(\bfA^\top)_{\cdot,j'}+\bfW_{j'}.
\]
Let \(I_{j'}\) denote the set of input scales retained in the final estimation
region \(\mcE_N\), and let \(\widetilde I_{j'}\supseteq I_{j'}\) denote the
larger set of input scales used in the regression region \(\mcR_N\).

There are two natural procedures.

\paragraph{Direct regression.}
One may regress only on the variables in \(I_{j'}\):
\[
(\widehat{\bfA}^{\top,\mathrm{dir}})_{I_{j'},j'}
:=
(\bfU_{I_{j'}}^\top \bfU_{I_{j'}})^{-1}
\bfU_{I_{j'}}^\top \bfF_{j'} .
\]
Substituting the block-column model gives
\[
(\widehat{\bfA}^{\top,\mathrm{dir}}-\bfA^\top)_{I_{j'},j'}
=
(\bfU_{I_{j'}}^\top \bfU_{I_{j'}})^{-1}
\bfU_{I_{j'}}^\top
\bfU_{I_{j'}^c}
(\bfA^\top)_{I_{j'}^c,j'}
+
(\bfU_{I_{j'}}^\top \bfU_{I_{j'}})^{-1}
\bfU_{I_{j'}}^\top \bfW_{j'} .
\]
Thus the omitted-variable bias is generated by all coefficients outside
\(I_{j'}\).

\paragraph{Nested-support regression.}

Instead, our estimator regresses on the enlarged set
\(\widetilde I_{j'}\), and then keeps only the entries in \(I_{j'}\):
\[
(\widehat{\bfA}^{\top})_{\widetilde I_{j'},j'}
:=
(\bfU_{\widetilde I_{j'}}^\top \bfU_{\widetilde I_{j'}})^{-1}
\bfU_{\widetilde I_{j'}}^\top \bfF_{j'} .
\]
Then
\[
(\widehat{\bfA}^{\top}-\bfA^\top)_{\widetilde I_{j'},j'}
=
(\bfU_{\widetilde I_{j'}}^\top \bfU_{\widetilde I_{j'}})^{-1}
\bfU_{\widetilde I_{j'}}^\top
\bfU_{\widetilde I_{j'}^c}
(\bfA^\top)_{\widetilde I_{j'}^c,j'}
+
(\bfU_{\widetilde I_{j'}}^\top \bfU_{\widetilde I_{j'}})^{-1}
\bfU_{\widetilde I_{j'}}^\top \bfW_{j'} .
\]
After restricting to \(I_{j'}\), the omitted-variable bias now only involves
the coefficients outside the larger set \(\widetilde I_{j'}\), rather than
outside \(I_{j'}\). Since \(\widetilde I_{j'}\supseteq I_{j'}\), this removes
the contribution of the intermediate scales
\(\widetilde I_{j'}\setminus I_{j'}\), which is precisely the gain used in the omitted-variable bias estimate.

It remains to explain why the enlarged regression does not change the variance
order. Write
\[
G_1:=\bfU_{I_{j'}}^\top\bfU_{I_{j'}},
\qquad
G_2:=\bfU_{\widetilde I_{j'}}^\top\bfU_{\widetilde I_{j'}} .
\]
Conditionally on \(\bfU\), the covariance of the direct-regression noise is
controlled by \(G_1^{-1}\), while the covariance of the nested-regression
noise, restricted back to \(I_{j'}\), is controlled by
\[
(G_2^{-1})_{I_{j'},I_{j'}} .
\]
Writing \(\widetilde I_{j'}=I_{j'}\cup K_{j'}\), the Schur complement formula
gives
\begin{align}\label{eq:Schur}
(G_2^{-1})_{I_{j'},I_{j'}}
=
\left(
G_1-G_{12}G_{22}^{-1}G_{21}
\right)^{-1}
\succeq G_1^{-1}.
\end{align}
Thus the nested regression indeed has larger variance in Loewner order.

The preceding comparison is deterministic. To see that the variance increase is only by a constant factor, we use the good design event obtained by combining the population diagonal preconditioning in Lemma~\ref{lemma:property_UAW} with Gaussian sample covariance concentration. Namely, on this high probability event,
\[
cN\bfD_{\widetilde I_{j'}}^{-2r_1}
\preceq
G_2
\preceq
CN\bfD_{\widetilde I_{j'}}^{-2r_1}.
\]
Since \(\bfD^{r_1}\) is diagonal, coordinate restriction commutes with the
preconditioner: for any \(I\subseteq \widetilde I_{j'}\),
\[
\bigl(\bfD_{\widetilde I_{j'}}^{r_1} M
\bfD_{\widetilde I_{j'}}^{r_1}\bigr)_{I,I}
=
\bfD_I^{r_1} M_{I,I}\bfD_I^{r_1}.
\]
Therefore,
\begin{align}\label{eq:G2}
(G_2^{-1})_{I_{j'},I_{j'}}
\preceq
\frac1{cN}\bfD_{I_{j'}}^{2r_1}.
\end{align}

The same good design event, restricted to the coordinates \(I_{j'}\), gives
\[
cN\bfD_{I_{j'}}^{-2r_1}
\preceq
G_1
\preceq
CN\bfD_{I_{j'}}^{-2r_1}.
\]
Consequently,
\begin{align}\label{eq:G1}
G_1^{-1}
\succeq
\frac1{CN}\bfD_{I_{j'}}^{2r_1}.
\end{align}
Combining \eqref{eq:G2} and \eqref{eq:G1} yields
\[
(G_2^{-1})_{I_{j'},I_{j'}}
\preceq
\frac{C}{c}G_1^{-1}.
\]
Together with \eqref{eq:Schur}, we obtain
\[
G_1^{-1}
\preceq
(G_2^{-1})_{I_{j'},I_{j'}}
\preceq
\frac{C}{c}G_1^{-1}.
\]
Hence nested regression may increase the variance, but only by a universal
constant factor. The variance order is unchanged, while the omitted-variable
bias is reduced from the contribution of \(I_{j'}^c\) to that of
\(\widetilde I_{j'}^c\).

In short, nested-support regression improves the bias because it moves scales in
\(\widetilde I_{j'}\setminus I_{j'}\) from the omitted part into the fitted
model. It does not change the variance order because the inverse Gram matrix is
controlled by a restriction-compatible diagonal preconditioner. The key point
is not that enlarging a regression support is harmless in general; rather, it
is harmless here up to constants because the covariance geometry is diagonal in
the multiscale coordinates.

In Appendix~\ref{app:diagonal_input}, we show that if the input covariance
matrix \(\bfC_u\) is exactly diagonal in the wavelet basis, then the
nested-support regression step is unnecessary. In that case, the same estimator
defined in \eqref{eq:estimator_column_box} and \eqref{eq:estimator_operator_box},
with the choice \(\widetilde J(j')\equiv J(j')\), also achieves the minimax
rate \(N^{-\gamma}\) up to logarithmic factors.
\end{remark}

\subsection{Minimax lower bound}\label{subsec:minimax_lower}

The following proposition establishes the minimax lower bound for learning
linear operators under the operator norm \(\|\cdot\|_{H^t\to H^{-t'}}\).

\begin{proposition}[Minimax lower bound]\label{prop:lower_bound}
Under the setting of Theorem~\ref{thm:main1},
\begin{align*}
\inf_{\widehat{\mcA}}\sup_{\mcA\in \mcO(s,s',R)}
\E \, \|\widehat{\mcA}-\mcA\|_{H^{t}\to H^{-t'}}
\,\gtrsim\,
N^{-\min\bigl\{
\frac12,\,
\frac{t-s}{2(r_1-s)+d},\,
\frac{t'-s'}{(2(-r_2-s')+d)_{+}}
\bigr\}}.
\end{align*}
\end{proposition}

The proof proceeds by reducing the problem to the equivalent matrix model and then applying Assouad's lemma via a unified thin-strip construction. The latter is used in three dyadic configurations ---coarse-to-coarse, fine-to-coarse, coarse-to-fine--- which produce the three contributions appearing in the lower bound. We begin by recalling Assouad's lemma.
 
\begin{lemma}[Assouad's lemma {\cite[Lemma 24.3]{van2000asymptotic}, \cite[Theorem 2.12]{Tsybakov2009}}]\label{lemma:assouad}
Let $\Theta=\{0,1\}^m$, and let $T$ be an estimator based on an observation whose distribution belongs to the family $\{P_\theta:\theta\in\Theta\}$. Let $\psi:\Theta\to\mathcal X$ be a parameter map into a metric space  $(\mathcal{X}, \mathsf{d})$, and let
\[
H(\theta,\theta'):=\sum_{j=1}^m \mathbf 1_{\{\theta_j\neq \theta_j'\}}
\]
denote the Hamming distance on $\Theta$. Then
\[
\max_{\theta\in\Theta}\E_\theta \mathsf{d}\bigl(T,\psi(\theta)\bigr)
\ge
\min_{H(\theta,\theta')\ge 1}
\frac{\mathsf{d}\bigl(\psi(\theta),\psi(\theta')\bigr)}{H(\theta,\theta')}
\cdot
\frac{m}{4}
\cdot
\min_{H(\theta,\theta')=1}\bigg(1-\sqrt{\frac{\dkl(P_{\theta}\|P_{\theta'})}{2}}\bigg),
\]
where $\dkl(P\|Q)$ denotes the Kullback-Leibler divergence between two probability measures $P$ and $Q$. 
\end{lemma}

We next explain how Assouad's lemma will be used in our setting. 
By the equivalent matrix formulation of the infinite-dimensional learning
problem, it suffices to prove a lower bound in the matrix model
\[
\bff_i=\bfA\bfu_i+\bfw_i,\qquad 1\le i\le N,
\]
under the loss
\[
\mathsf d(\bfA,\bfB)
:=
\|\bfD^{-t'}(\bfA-\bfB)\bfD^{-t}\|.
\]
We will construct finite families
\[
\{\bfA_\theta:\theta\in\Theta\}\subseteq \mcM(s,s',R'),
\qquad
\Theta=\{0,1\}^m,
\]
where the value of \(m\) and the support of the perturbations will be chosen
according to the relevant dyadic regime. The parameter map in Assouad's lemma is
simply
\[
\psi(\theta)=\bfA_\theta .
\]
The perturbations will be supported on thin strips inside dyadic blocks. This
choice keeps the operator-norm separation transparent while making the Kullback-Leibler (KL) divergence between neighboring hypotheses easy to control. For each \(\bfA\),
we write \(P_{\bfA}\) for the joint law of the data under the matrix model; the
next lemma gives the KL calculation used throughout.

\begin{lemma}[KL divergence in the matrix model]\label{lemma:KL}
For \(\bfA_0,\bfA_1\in\mcM(s,s',R')\), let \(P_{\bfA_0}\) and
\(P_{\bfA_1}\) denote the joint laws of
\(\{(\bfu_i,\bff_i)\}_{i=1}^N\) under the matrix model. Then
\[
\dkl(P_{\bfA_1}\|P_{\bfA_0})
\asymp
\frac{N}{2}
\|\bfD^{r_2}(\bfA_1-\bfA_0)\bfD^{-r_1}\|_{\mathrm{HS}}^2,
\]
whenever the right-hand side is finite.
\end{lemma}

\begin{proof}%[Proof of Lemma \ref{lemma:KL}]
For \(\bfA\in\mcM(s,s',R')\), write
\[
P_{\bfA}:=\mcL_{\bfA}\bigl(\{(\bfu_i,\bff_i)\}_{i=1}^N\bigr).
\]
Conditionally on \(\bfu_i\),
\[
\bff_i\mid \bfu_i\sim \mcN(\bfA\bfu_i,\bfC_w),
\]
and hence
\[
P_{\bfA}
=
\prod_{i=1}^N P_{\bfu}(\bfu_i)\,P_{\bfA}(\bff_i\mid \bfu_i).
\]
Therefore,
\begin{align*}
\dkl(P_{\bfA_1}\|P_{\bfA_0})
&=
\E_{\bfU}\Bigl[
\dkl(P^{\bfF\mid\bfU}_{\bfA_1}\|P^{\bfF\mid\bfU}_{\bfA_0})
\Bigr]  \\
&=
\frac12
\E\left[
\sum_{i=1}^N
(\bfA_1\bfu_i-\bfA_0\bfu_i)^\top
\bfC_w^{-1}
(\bfA_1\bfu_i-\bfA_0\bfu_i)
\right] \\
&=
\frac{N}{2}\tr\Bigl(
\bfC_w^{-1}(\bfA_1-\bfA_0)\bfC_u(\bfA_1-\bfA_0)^\top
\Bigr).
\end{align*}
Using the covariance equivalences in Lemma~\ref{lemma:property_UAW},
\[
\bfC_u\asymp \bfD^{-2r_1},
\qquad
\bfC_w^{-1}\asymp \bfD^{2r_2},
\]
we obtain
\[
\dkl(P_{\bfA_1}\|P_{\bfA_0})
\asymp
\frac{N}{2}
\tr\Bigl(
\bfD^{r_2}(\bfA_1-\bfA_0)
\bfD^{-2r_1}
(\bfA_1-\bfA_0)^\top
\bfD^{r_2}
\Bigr)=\frac{N}{2}
\|\bfD^{r_2}(\bfA_1-\bfA_0)\bfD^{-r_1}\|_{\mathrm{HS}}^2.
\]
In the last comparison we used the inequality
\[
\sigma_{\min}(\bfA)\tr(\bfB)\le \tr(\bfA\bfB)=\tr(\bfB\bfA)\le \sigma_{\max}(\bfA)\tr(\bfB),
\]
valid for symmetric positive semidefinite matrices \(\bfA\) and \(\bfB\).
This proves the claim.
\end{proof}

The next lemma implements the preceding Assouad strategy on a thin strip inside
one dyadic block. It simultaneously covers row-type and column-type
perturbations.

\begin{lemma}[Thin-strip Assouad construction]\label{lem:thin_strip_assouad}
Let $J_{\mathrm{out}},J_{\mathrm{in}}\in\mathbb {N}_0$, and let
\[
\mathcal{I}\subseteq \Lambda_{J_{\mathrm{out}}},
\qquad
\mathcal{J} \subseteq \Lambda_{J_{\mathrm{in}}}
\]
be finite sets satisfying
\[
\min\{|\mathcal I|,|\mathcal J|\}=1.
\]
Set
\[
m:=|\mathcal I|\,|\mathcal J|=\max\{|\mathcal I|,|\mathcal J|\}.
\]
For $\theta=(\theta_{\lambda,\lambda'})_{(\lambda,\lambda')\in \mathcal I\times\mathcal J}\in\{0,1\}^{m}$ and $(\lambda,\lambda') \in \Lambda \times \Lambda$, define
\[
(\bfA_\theta)_{\lambda,\lambda'}
=
\delta\,\theta_{\lambda,\lambda'}\,\mathbf 1_{\{(\lambda,\lambda')\in\mathcal I\times\mathcal J\}},
\]
where $\delta>0$.

Assume that
\begin{equation}\label{eq:delta_thin_strip}
\delta
\le
c_0\min\Bigl\{
2^{J_{\mathrm{out}}s'}2^{J_{\mathrm{in}}s}m^{-1/2},
\;
N^{-1/2}2^{-J_{\mathrm{out}}r_2}2^{J_{\mathrm{in}}r_1}
\Bigr\},
\end{equation}
for a sufficiently small constant $c_0>0$. Then:

\begin{enumerate}
\item For every $\theta\in\{0,1\}^m$, one has $\bfA_\theta\in\mcM(s,s',R')$.

\item For every $\theta,\theta'\in\{0,1\}^m$ with $H(\theta,\theta')=1$,
\[
\dkl(P_{\bfA_\theta}\|P_{\bfA_{\theta'}})\lesssim 1.
\]

\item For every $\theta,\theta'\in\{0,1\}^m$,
\[
\|\bfD^{-t'}(\bfA_\theta-\bfA_{\theta'})\bfD^{-t}\|
\asymp
\delta\,2^{-J_{\mathrm{out}}t'}2^{-J_{\mathrm{in}}t}\sqrt{H(\theta,\theta')}.
\]

\item Consequently,
\begin{align}
\inf_{\widehat{\bfA}}\sup_{\bfA\in\mcM(s,s',R')}
\E \, \|\bfD^{-t'}(\widehat{\bfA}-\bfA)\bfD^{-t}\|
\,\gtrsim\,
\delta\,2^{-J_{\mathrm{out}}t'}2^{-J_{\mathrm{in}}t}\sqrt{m}.
\label{eq:thin_strip_general}
\end{align}
In particular, after optimizing over $\delta$ subject to \eqref{eq:delta_thin_strip},
{\scriptsize
\begin{align}
\inf_{\widehat{\bfA}}\sup_{\bfA\in\mcM(s,s',R')}
\E\, \|\bfD^{-t'}(\widehat{\bfA}-\bfA)\bfD^{-t}\|
\,\gtrsim\,
\min\Bigl\{
N^{-1/2}2^{-J_{\mathrm{out}}(r_2+t')}2^{J_{\mathrm{in}}(r_1-t)}\sqrt{m},
\;
2^{-J_{\mathrm{out}}(t'-s')}2^{-J_{\mathrm{in}}(t-s)}
\Bigr\}.
\label{eq:thin_strip_rate}
\end{align}
}
\end{enumerate}
\end{lemma}

\begin{proof}

We verify the four claims in turn.

\smallskip
\noindent
\paragraph{Step 1: parameter space constraint.}
Since one of $\mathcal I,\mathcal J$ is a singleton, the matrix $\bfA_\theta$ is either a row vector or a column vector supported inside the block $\mathcal I\times\mathcal J$. Hence
\[
\|\bfA_\theta\|\le \delta\sqrt{m}.
\]
After inserting the Sobolev weights, every active entry acquires the same factor $2^{-J_{\mathrm{out}}s'}2^{-J_{\mathrm{in}}s}$, so
\[
\|\bfD^{-s'}\bfA_\theta\bfD^{-s}\|
\lesssim
\delta\,2^{-J_{\mathrm{out}}s'}2^{-J_{\mathrm{in}}s}\sqrt{m}.
\]
Thus the first condition in \eqref{eq:delta_thin_strip} guarantees that $\bfA_\theta\in\mcM(s,s',R')$.

\smallskip
\noindent
\paragraph{Step 2: KL control.}
If $H(\theta,\theta')=1$, then $\bfA_\theta-\bfA_{\theta'}$ has exactly one nonzero entry, equal to $\pm\delta$, at some
\[
(\lambda_0,\lambda_0')\in \mathcal I\times\mathcal J
\subseteq \Lambda_{J_{\mathrm{out}}}\times\Lambda_{J_{\mathrm{in}}}.
\]
By Lemma \ref{lemma:KL},
\[
\dkl(P_{\bfA_\theta}\|P_{\bfA_{\theta'}})
\asymp
\frac{N}{2}
\|\bfD^{r_2}(\bfA_\theta-\bfA_{\theta'})\bfD^{-r_1}\|_{\mathrm{HS}}^2.
\]
Since only one entry is active,
\[
\|\bfD^{r_2}(\bfA_\theta-\bfA_{\theta'})\bfD^{-r_1}\|_{\mathrm{HS}}^2
\asymp
\delta^2\,2^{2J_{\mathrm{out}}r_2}2^{-2J_{\mathrm{in}}r_1},
\]
and therefore
\[
\dkl(P_{\bfA_\theta}\|P_{\bfA_{\theta'}})
\asymp
N\delta^2\,2^{2J_{\mathrm{out}}r_2}2^{-2J_{\mathrm{in}}r_1}.
\]
Hence the second condition in \eqref{eq:delta_thin_strip} implies
\[
\max_{H(\theta,\theta')=1}\dkl(P_{\bfA_\theta}\|P_{\bfA_{\theta'}})\lesssim 1.
\]

\smallskip
\noindent
\paragraph{Step 3: separation in the loss metric.}
For arbitrary $\theta,\theta'$, the difference $\bfA_\theta-\bfA_{\theta'}$ is supported on the same strip and has exactly $H(\theta,\theta')$ nonzero entries, each equal to $\pm\delta$. Since it is again either a row vector or a column vector,
\[
\|\bfA_\theta-\bfA_{\theta'}\|
=
\delta\sqrt{H(\theta,\theta')}.
\]
Applying the weights $\bfD^{-t'}$ and $\bfD^{-t}$ gives
\[
\|\bfD^{-t'}(\bfA_\theta-\bfA_{\theta'})\bfD^{-t}\|
\asymp
\delta\,2^{-J_{\mathrm{out}}t'}2^{-J_{\mathrm{in}}t}\sqrt{H(\theta,\theta')}.
\]
This proves part (3). Since $H(\theta,\theta')\le m$,
\[
\sqrt{H(\theta,\theta')}
\ge
m^{-1/2}H(\theta,\theta'),
\]
so
\[
\frac{\|\bfD^{-t'}(\bfA_\theta-\bfA_{\theta'})\bfD^{-t}\|}{H(\theta,\theta')}
\gtrsim
\delta\,2^{-J_{\mathrm{out}}t'}2^{-J_{\mathrm{in}}t}m^{-1/2}.
\]

\smallskip
\noindent
\paragraph{Step 4: Assouad lower bound.}
Applying Lemma~\ref{lemma:assouad} with
\[
\psi(\theta)=\bfA_\theta,
\qquad
\mathsf d(\bfA,\bfB)=\|\bfD^{-t'}(\bfA-\bfB)\bfD^{-t}\|,
\]
and using Steps 2 and 3, we obtain
\[
\inf_{\widehat{\bfA}}\sup_{\bfA\in\mcM(s,s',R')}
\E \,\|\bfD^{-t'}(\widehat{\bfA}-\bfA)\bfD^{-t}\|
\gtrsim
\delta\,2^{-J_{\rm out}t'}2^{-J_{\rm in}t}\sqrt m.
\]
Substituting the largest admissible \(\delta\) in
\eqref{eq:delta_thin_strip} gives \eqref{eq:thin_strip_rate}.
\end{proof}

We now prove Proposition~\ref{prop:lower_bound} by applying the thin-strip
lemma in three dyadic regimes.

\begin{proof}[Proof of Proposition \ref{prop:lower_bound}]
As discussed above, it suffices to prove the
corresponding lower bound in the matrix model under the loss
\[
\mathsf d(\bfA,\bfB)
=
\|\bfD^{-t'}(\bfA-\bfB)\bfD^{-t}\|.
\]
We apply Lemma~\ref{lem:thin_strip_assouad} in three dyadic
configurations.

\paragraph{Case 1: coarse-to-coarse.}
Take
\[
J_{\mathrm{out}}=0,\qquad J_{\mathrm{in}}=0,\qquad m=1.
\]
Then \eqref{eq:thin_strip_rate} gives
\[
\inf_{\widehat{\bfA}}\sup_{\bfA\in\mcM(s,s',R')}
\E \|\bfD^{-t'}(\widehat{\bfA}-\bfA)\bfD^{-t}\| 
\,\gtrsim\,
\min\{N^{-1/2},1\}
\asymp
N^{-1/2}.
\]

\paragraph{Case 2: fine-to-coarse.}
Fix $J_1\in\mathbb N$ and choose a thin strip inside the block $(0,J_1)$. Then
\[
J_{\mathrm{out}}=0,\qquad J_{\mathrm{in}}=J_1,\qquad m\asymp 2^{J_1d}.
\]
Hence
\[
\inf_{\widehat{\bfA}}\sup_{\bfA\in\mcM(s,s',R')}
\E \,\|\bfD^{-t'}(\widehat{\bfA}-\bfA)\bfD^{-t}\| \,
\gtrsim\,
\min\Bigl\{
N^{-1/2}2^{J_1(r_1-t+d/2)},
\,
2^{-J_1(t-s)}
\Bigr\}.
\]
Since $r_1-s+d/2>0$, choose
\[
J_1=
\left\lceil
\frac{\log_2 N}{2(r_1-s+d/2)}
\right\rceil.
\]
Then the two terms are balanced, and
\[
N^{-1/2}2^{J_1(r_1-t+d/2)}
\asymp
2^{-J_1(t-s)}
\asymp
N^{-\frac{t-s}{2(r_1-s)+d}}.
\]

\paragraph{Case 3: coarse-to-fine.}
Fix $J_2\in\mathbb N$ and choose a thin strip inside the block $(J_2,0)$. Then
\[
J_{\mathrm{out}}=J_2,\qquad J_{\mathrm{in}}=0,\qquad m\asymp 2^{J_2d}.
\]
Thus
\[
\inf_{\widehat{\bfA}}\sup_{\bfA\in\mcM(s,s',R')}
\E \, \|\bfD^{-t'}(\widehat{\bfA}-\bfA)\bfD^{-t}\| 
\,\gtrsim\,
\min\Bigl\{
N^{-1/2}2^{J_2(-r_2-t'+d/2)},
\,
2^{-J_2(t'-s')}
\Bigr\}.
\]
If $-r_2-s'+d/2>0$, choose
\[
J_2=
\left\lceil
\frac{\log_2 N}{2(-r_2-s'+d/2)}
\right\rceil,
\]
which yields
\[
N^{-1/2}2^{J_2(-r_2-t'+d/2)}
\asymp
2^{-J_2(t'-s')}
\asymp
N^{-\frac{t'-s'}{2(-r_2-s')+d}}.
\]

Collecting the three cases, we conclude
\begin{align*}
\inf_{\widehat{\bfA}}\sup_{\bfA\in\mcM(s,s',R')}
\E \, \|\bfD^{-t'}(\widehat{\bfA}-\bfA)\bfD^{-t}\|
\, \gtrsim \,
\max\bigg\{
N^{-1/2},
\;
N^{-\frac{t-s}{2(r_1-s)+d}},
\;
N^{-\frac{t'-s'}{2(-r_2-s')+d}}
\bigg\},
\end{align*}
where the third term is included only when the corresponding denominator $2(-r_2-s')+d$ is positive. Equivalently,
\begin{align*}
\inf_{\widehat{\bfA}}\sup_{\bfA\in\mcM(s,s',R')}
\E \, \|\bfD^{-t'}(\widehat{\bfA}-\bfA)\bfD^{-t}\| \,
\gtrsim\,
N^{-\min\bigl\{
\frac12,\,
\frac{t-s}{2(r_1-s)+d},\,
\frac{t'-s'}{(2(-r_2-s')+d)_{+}}
\bigr\}}.
\end{align*}
This proves the proposition.
\end{proof}

\section{Computational cost}\label{sec:computation}

Proposition~\ref{prop:upper_bound} shows that the estimator constructed in
Subsection~\ref{subsec:minimax_upper} achieves the minimax rate
\(N^{-\gamma}\), up to logarithmic factors. We now study the computational
cost required to attain this statistical accuracy. To express the cost as a
function of a target error, we write
\[
\varepsilon=N^{-\gamma}.
\]
Throughout this section, \(\Cost(\varepsilon)\) denotes the number of arithmetic
operations required to construct the estimator from the finite coefficient data
used in Subsection~\ref{subsec:minimax_upper}, when the columnwise regressions
are solved by dense direct least-squares methods. We suppress polylogarithmic
factors in \(\varepsilon^{-1}\) and \(\delta^{-1}\).

This section has two parts. In Subsection~\ref{subsec:cost}, we derive the
cost exponent for the scale-adaptive estimator. In
Subsection~\ref{subsec:cost_optimality}, we explain why the two components of
this exponent are unavoidable within the present blockwise
dense-regression framework.

\subsection{Cost of the scale-adaptive estimator}\label{subsec:cost}

The following proposition gives the computational cost of the estimator in
\eqref{eq:estimator_operator_box}.

\begin{proposition}[Cost of the scale-adaptive estimator]\label{prop:cost}
Under the setting of Theorem~\ref{thm:main1}, fix a target accuracy
\( \varepsilon \in (0,1) \), and choose \( N \) so that
\begin{align}\label{eq:notation_transform}
\varepsilon = N^{-\gamma}.
\end{align}
Then the estimator in \eqref{eq:estimator_operator_box}, implemented by
columnwise dense least squares, satisfies
\[
\Cost(\varepsilon)
\lesssim_{\log}
\varepsilon^{-\kappa},
\]
where
\[
\kappa
=
\max\left\{
\kappa_{\rm in},
\kappa_{\rm out}
\right\},
\]
with
\[
\kappa_{\rm in}
:=
2
+
\frac{2d}{\min\{r_1,t\}-s}
+
\frac{\bigl(2(r_1-t)+d\bigr)_{+}}{t-s},
\]
and
\[
\kappa_{\rm out}
:=
\frac{d+\bigl(2(-r_2-s')+d\bigr)_{+}}{t'-s'}
=
\max\left\{
\frac{d}{t'-s'},\
\frac{2(-r_2-s')+2d}{t'-s'}
\right\}.
\]
\end{proposition}

The exponent in Proposition~\ref{prop:cost} has two components. The term
\(\kappa_{\rm in}\) reflects the input-side cost of solving regressions over
many input coordinates, including the enlarged regression support needed to
control omitted-variable bias. The term \(\kappa_{\rm out}\) reflects the
output-side cost of processing many response coordinates at high output
resolutions. The overall cost exponent is the larger of these two contributions.

The heterogeneous multiscale structure of the problem also implies that the
statistical and computational bottlenecks need not coincide. The following
remark gives a concrete regime where the minimax rate is input-side dominated,
while the computational cost is output-side dominated.

\begin{remark}[Statistical and computational bottlenecks may differ]\label{remark:stats_computation}
Intuitively, a large value of \(r_1\) makes
high-frequency input directions poorly excited. This can make the input side
statistically hardest: even recovering a single output component as a function
of many input frequencies requires many samples. Computationally, however, this
is still essentially one large single-response regression. By contrast, the
output side may be statistically easier but computationally more expensive:
recovering one low-frequency input component across a high output scale amounts
to estimating many response coordinates, or equivalently many scalar
regressions, since \(|\nabla_{j'}|\asymp 2^{j'd}\).

For example, take
\[
d=1,\qquad s=s'=0,\qquad r_1=t=2,\qquad r_2=0,\qquad t'=\frac12 .
\]
Then the minimax exponent in Theorem~\ref{thm:main1} is
\[
\gamma
=
\min\left\{
\frac12,\,
\frac{2}{2\cdot 2+1},\,
\frac{1/2}{1}
\right\}
=
\frac25 ,
\]
so the statistical rate is dominated by the input-side contribution. On the
other hand, Proposition~\ref{prop:cost} gives
\[
\kappa_{\rm in}
=
2+\frac{2}{2}+\frac{1}{2}
=
\frac72,
\qquad
\kappa_{\rm out}
=
\frac{1+1}{1/2}
=
4.
\]
Therefore the arithmetic cost is dominated by the output-side exponent, even
though the minimax rate is dominated by the input side.
\end{remark}

\begin{proof}[Proof of Proposition \ref{prop:cost}]

Recall from Definition \ref{def:estimation_region} that the minimal estimation region is
\[
j(t-s)+j'(t'-s')\le \gamma \log_2 N=\log_2(1/\varepsilon),
\]
and for $j'=0,1,\dots,\big\lceil \frac{\gamma\log_2 N}{t'-s'} \big\rceil$,
\[
J(j')
=
\left\lceil
\frac{\log_2(1/\varepsilon)-j'(t'-s')}{t-s}
\right\rceil,\qquad
\widetilde J(j')
=
\left\lceil
\frac{\log_2(1/\varepsilon)-j'(t'-s')}
{\min\{r_1,t\}-s}
\right\rceil.
\]
For the \(j'\)-th column-block regression, we use \(N_{j'}\) samples, where
\begin{align}\label{eq:Njprime_2}
N_{j'}
&\asymp
\max\left\{2^{\widetilde J(j')d}+\log((\log N)/\delta), \
N^{2\gamma}
\sup_{0\le j\le J(j')}
2^{-2jt-2j't'+2jr_1-2j'r_2+d\max\{j,j'\}}\right\}\nonumber\\
&\asymp 2^{\widetilde J(j')d}+ \log(\log (1/\varepsilon)/\delta)+
\varepsilon^{-2}
\sup_{0\le j\le J(j')}
2^{-2jt-2j't'+2jr_1-2j'r_2+d\max\{j,j'\}}.
\end{align}
Here we used the change of variables \(\varepsilon=N^{-\gamma}\), and the
equivalence between the maximum and the sum for nonnegative quantities. We solve the corresponding problem by dense least squares, computing
\[
(\widehat{\bfA}^{\top})_{\widetilde I_{j'},j'}
=
\big((\bfU^{(j')}_{\widetilde I_{j'}})^{\top}
\bfU^{(j')}_{\widetilde I_{j'}}\big)^{-1}
(\bfU^{(j')}_{\widetilde I_{j'}})^{\top}
\bfF^{(j')}_{j'}.
\]
The Gram-matrix computation and the multi-response multiplication cost
\[
\mathcal{O}\left(
N_{j'}2^{2\widetilde J(j')d}
+
N_{j'}2^{\widetilde J(j')d+j'd}
\right),
\]
with the matrix inversion cost absorbed since
\(N_{j'}\gtrsim 2^{\widetilde J(j')d}\).
Hence the total computational cost satisfies
\begin{align*}
\Cost(\varepsilon)
&\asymp
\sum_{j'=0}^{\frac{\log_2(1/\varepsilon)}{t'-s'}}
N_{j'}
\left(
2^{2\widetilde J(j')d}
+
2^{\widetilde J(j')d+j'd}
\right)\\
&\lesssim_{\log} \underbrace{\sum_{j'=0}^{\frac{\log_2(1/\varepsilon)}{t'-s'}}
2^{3\widetilde J(j')d}}_{=:\circnum{1}}
+ \underbrace{\sum_{j'=0}^{\frac{\log_2(1/\varepsilon)}{t'-s'}}
2^{2\widetilde J(j')d+j'd}}_{=:\circnum{2}}\\
&\quad + \underbrace{\varepsilon^{-2}
\sum_{j'=0}^{\frac{\log_2(1/\varepsilon)}{t'-s'}}
2^{2\widetilde J(j')d}
\sup_{0\le j\le J(j')}
2^{-2jt-2j't'+2jr_1-2j'r_2+d\max\{j,j'\}}}_{=:\circnum{3}}\\
&\quad + \underbrace{\varepsilon^{-2}
\sum_{j'=0}^{\frac{\log_2(1/\varepsilon)}{t'-s'}}
2^{\widetilde J(j')d+j'd}
\sup_{0\le j\le J(j')}
2^{-2jt-2j't'+2jr_1-2j'r_2+d\max\{j,j'\}}}_{=:\circnum{4}}
.
\end{align*}

For the first term, substituting the definition of \(\widetilde J(j')\) gives
\[
\circnum{1}=\sum_{j'=0}^{\frac{\log_2(1/\varepsilon)}{t'-s'}}
2^{
3d
\frac{\log_2(1/\varepsilon)-j'(t'-s')}
{\min\{r_1,t\}-s}
}
\lesssim
\varepsilon^{-\frac{3d}{\min\{r_1,t\}-s}}.
\]
Similarly, for the second term,
\[
\circnum{2}=\sum_{j'=0}^{\frac{\log_2(1/\varepsilon)}{t'-s'}}
2^{
2d
\frac{\log_2(1/\varepsilon)-j'(t'-s')}
{\min\{r_1,t\}-s}
+j'd
}\lesssim_{\log}  \varepsilon^{-\frac{2d}{\min\{r_1,t\}-s}}+ 
\varepsilon^{
-
\frac{d}{t'-s'}
}.
\]

For the third term,
\begin{align*}
\circnum{3}&
=\varepsilon^{-2}
\sum_{j'=0}^{\frac{\log_2(1/\varepsilon)}{t'-s'}}
2^{2\widetilde J(j')d}
\sup_{0\le j\le J(j')}
2^{-2jt-2j't'+2jr_1-2j'r_2+d\max\{j,j'\}}\\
&=\varepsilon^{-2}
\sum_{j'=0}^{\frac{\log_2(1/\varepsilon)}{t'-s'}}
\sup_{0\le j\le J(j')}
2^{2d \frac{\log_2(1/\varepsilon)-j'(t'-s')}
{\min\{r_1,t\}-s}+2j(r_1-t)-2j'(t'+r_2)+d\max\{j,j'\}}
\end{align*}
To obtain an estimate up to logarithmic factors, it suffices to maximize the exponent
\[
2d
\frac{\log_2(1/\varepsilon)-j'(t'-s')}
{\min\{r_1,t\}-s}
+
2j(r_1-t)
-2j'(t'+r_2)
+d\max\{j,j'\}
\]
over the triangular region
\[
j\ge 0,\qquad j'\ge 0,\qquad
j(t-s)+j'(t'-s')\le \log_2(1/\varepsilon).
\]
The exponent is piecewise affine, with the only change of affine form occurring
across the line \(j=j'\). Hence it suffices to check the vertices of the
triangular region together with the possible kink point on \(j=j'\). The kink
point lies on the edge
\[
j(t-s)+j'(t'-s')=\log_2(1/\varepsilon).
\]
Along this edge, the term \(d\max\{j,j'\}\) is a convex piecewise-affine
function of one variable, so its maximum is attained at one of the two endpoints
of the edge. Thus the kink point is dominated by one of the endpoints. It
therefore suffices to check the three vertices
\[
(0,0),\qquad
\left(\frac{\log_2(1/\varepsilon)}{t-s},0\right),
\qquad
\left(0,\frac{\log_2(1/\varepsilon)}{t'-s'}\right).
\]
At the vertex \((0,0)\), the exponent equals
\[
\frac{2d\log_2(1/\varepsilon)}{\min\{r_1,t\}-s}.
\]
After multiplying by \(\varepsilon^{-2}\), this gives the contribution
\[
\varepsilon^{
-\left(
2+\frac{2d}{\min\{r_1,t\}-s}
\right)
}.
\]
At the vertex $\big(\frac{\log_2(1/\varepsilon)}{t-s},0\big)$,
the exponent equals
\[
\frac{2d\log_2(1/\varepsilon)}{\min\{r_1,t\}-s}
+
\frac{\bigl(2(r_1-t)+d\bigr)\log_2(1/\varepsilon)}{t-s}.
\]
If \(2(r_1-t)+d<0\), this vertex is smaller than the previous one. Hence the
two input-side vertices together give
\[
\varepsilon^{
-\big(
2+
\frac{2d}{\min\{r_1,t\}-s}
+
\frac{(2(r_1-t)+d)_{+}}{t-s}
\big)
}.
\]
At the vertex $\big(0,\frac{\log_2(1/\varepsilon)}{t'-s'}\big)$, the exponent becomes
\[
\frac{\bigl(-2t'-2r_2+d\bigr)\log_2(1/\varepsilon)}
{t'-s'}.
\]
After multiplying by \(\varepsilon^{-2}\), this gives
\[
\varepsilon^{
-\big(
2+
\frac{-2t'-2r_2+d}{t'-s'}
\big)
}
=
\varepsilon^{
-\frac{2(-r_2-s')+d}{t'-s'}
}.
\]
Therefore,
\[
\circnum{3}\lesssim \varepsilon^{
-\big(
2+
\frac{2d}{\min\{r_1,t\}-s}
+
\frac{(2(r_1-t)+d)_{+}}{t-s}
\big)
}+\varepsilon^{
-\frac{2(-r_2-s')+d}{t'-s'}
}.
\]

It remains to estimate \(\circnum{4}\). Using the definition of
\(\widetilde J(j')\), we have
\begin{align*}
\circnum{4}=\varepsilon^{-2}
\sum_{j'=0}^{\frac{\log_2(1/\varepsilon)}{t'-s'}}\sup_{0\le j\le J(j')} 2^{d
\frac{\log_2(1/\varepsilon)-j'(t'-s')}
{\min\{r_1,t\}-s}
+j'd
+
2j(r_1-t)
-2j'(t'+r_2)
+d\max\{j,j'\}}.
\end{align*}
The exponent has the same piecewise-affine structure as in the estimate of
\(\circnum{3}\). Therefore the same vertex argument over the triangular region gives
\[
\circnum{4}\lesssim_{\log } \varepsilon^{
-\big(
2+
\frac{d}{\min\{r_1,t\}-s}
+
\frac{(2(r_1-t)+d)_{+}}{t-s}
\big)
}+\varepsilon^{
-\frac{2(-r_2-s')+2d}{t'-s'}
}.
\]

Combining the above estimates, we obtain
\begin{align}\label{eq:cost_aux1}
\Cost(\varepsilon)
&\lesssim_{\log} \circnum{1}+\circnum{2}+\circnum{3} +\circnum{4}  \nonumber \\
&\lesssim_{\log} \varepsilon^{-\frac{3d}{\min\{r_1,t\}-s}} + 
\varepsilon^{
-
\frac{d}{t'-s'}
}+\varepsilon^{
-\big(
2+
\frac{2d}{\min\{r_1,t\}-s}
+
\frac{(2(r_1-t)+d)_{+}}{t-s}
\big)
}+\varepsilon^{
-\frac{2(-r_2-s')+2d}{t'-s'}
}.
\end{align}
Since \(0<\varepsilon<1\), only the largest exponent matters. The first term
in \eqref{eq:cost_aux1} is absorbed by the third term. Indeed, this follows
from
\[
2+
\frac{2d}{\min\{r_1,t\}-s}
+
\frac{(2(r_1-t)+d)_{+}}{t-s}>\frac{3d}{\min\{r_1,t\}-s},
\]
or equivalently,
\[
2
+
\frac{(2(r_1-t)+d)_{+}}{t-s}>\frac{d}{\min\{r_1,t\}-s}.
\]
If $r_1\ge t$, then
\[
2
+
\frac{(2(r_1-t)+d)_{+}}{t-s}=2
+
\frac{2(r_1-t)+d}{t-s}>\frac{d}{t-s}=\frac{d}{\min\{r_1,t\}-s}.
\]
If $r_1<t$, then, by the assumption $r_1-s>d/2$,
\[
2
+
\frac{(2(r_1-t)+d)_{+}}{t-s} \ge 2>\frac{d}{r_1-s}=\frac{d}{\min\{r_1,t\}-s}.
\]

Consequently, after absorbing the first term into the third and combining the
two output-side exponents, \eqref{eq:cost_aux1} simplifies to
\[
\Cost(\varepsilon)
\lesssim_{\log}
\varepsilon^{-\kappa},
\]
where
\[
\kappa
=
\max\left\{
\kappa_{\rm in},
\kappa_{\rm out}
\right\},
\]
with
\[
\kappa_{\rm in}
:=
2
+
\frac{2d}{\min\{r_1,t\}-s}
+
\frac{\bigl(2(r_1-t)+d\bigr)_{+}}{t-s},
\]
and
\[
\kappa_{\rm out}
:=
\frac{d+\bigl(2(-r_2-s')+d\bigr)_{+}}{t'-s'}
=
\max\left\{
\frac{d}{t'-s'},\ 
\frac{2(-r_2-s')+2d}{t'-s'}
\right\}.
\]
This completes the proof.
\end{proof}

\subsection{Optimality within the dense-regression framework}
\label{subsec:cost_optimality}

We next explain why the two exponents in Proposition~\ref{prop:cost} are
unavoidable within the present blockwise dense-regression framework.
The input-side barrier appears already in the single-column problem for
\(\bfA^\top\), which reduces to an infinite-dimensional single-response
regression problem; see \S~\ref{subsubsec:input}. The output-side barrier
appears already in the single-row problem for \(\bfA^\top\), which reduces to
estimating many response coordinates sharing a single input coordinate; see
\S~\ref{subsubsec:output}.

\subsubsection{Input-side regression barrier}\label{subsubsec:input}
Recall that, in the transposed representation \(\bfA^\top\), the row index is
an input wavelet index and the column index is an output wavelet index. To isolate the input-side obstruction, fix a single low-frequency output index
\(\lambda'_0\in\Lambda_0=\{(j,k):j=0,k\in\nabla_0\}\), and consider the subclass of matrices satisfying
\[
(\bfA^\top)_{\lambda,\mu}=0
\qquad\text{unless } \mu=\lambda'_0 .
\]
Thus the only unknown coefficients form the infinite column vector
\[
\mathbf{x}
:=
\bigl((\bfA^\top)_{\lambda,\lambda'_0}\bigr)_{\lambda\in\Lambda}.
\]
On this subclass, the model reduces to the infinite-dimensional
single-response regression problem
\[
\bfF_{\lambda'_0}
=
\bfU\mathbf{x}+\bfW_{\lambda'_0}.
\]

For this restricted problem, since \(|\lambda'_0|=0\), the contribution to the
weighted operator norm is the weighted \(\ell^2\)-error of the column vector:
\[
\|\bfD^{-t}(\widehat{\bfA}^\top-\bfA^\top)\bfD^{-t'}\|
=
\bigg(
\sum_{\lambda\in\Lambda}
2^{-2|\lambda|t}
|\widehat \bfx_\lambda-\bfx_\lambda|^2
\bigg)^{1/2}.
\]
On the other hand, the operator-class constraint gives
\[
\bigg(
\sum_{\lambda\in\Lambda}
2^{-2|\lambda|s}
|\bfx_\lambda|^2
\bigg)^{1/2}
\lesssim 1.
\]
Thus, at input scale \(j\), an admissible vector \(\bfx_j\) may have Euclidean
size
\[
\|\bfx_j\|_2\lesssim 2^{js},
\]
and its contribution to the loss is at most
\[
2^{-jt}\|\bfx_j\|_2
\lesssim
2^{-j(t-s)}.
\]
Therefore, input scales satisfying
\[
2^{-j(t-s)}\lesssim \varepsilon
\]
can be discarded without affecting the target accuracy. The finest input scale
that must be retained satisfies
\[
J_0
\asymp
\frac{\log_2(1/\varepsilon)}{t-s}.
\]
We therefore set
\[
J_0
:=
\left\lceil
\frac{\log_2(1/\varepsilon)}{t-s}
\right\rceil,
\qquad
I_0:=\{0,1,\ldots,J_0\}.
\]

However, direct regression only on \(\Lambda_{I_0}\) is not uniformly stable
over the full operator class. The coefficients outside \(I_0\) still enter the
response through the empirical design and create omitted-variable bias. To
control this effect, one must regress over an enlarged support. Namely, set
\[
\widetilde J_0
:=
\left\lceil
\frac{\log_2(1/\varepsilon)}{\min\{r_1,t\}-s}
\right\rceil,
\qquad
\widetilde I_0:=\{0,1,\ldots,\widetilde J_0\}.
\]
Then the single-column model can be written as
\[
\bfF_{\lambda'_0}
=
\bfU_{\widetilde I_0}\mathbf{x}_{\widetilde I_0}
+
\bfU_{\widetilde I_0^c}\mathbf{x}_{\widetilde I_0^c}
+
\bfW_{\lambda'_0}.
\]
The estimator solves the least-squares problem over
\(\Lambda_{\widetilde I_0}\) and then retains only the coordinates in
\(\Lambda_{I_0}\).

The enlarged regression support has dimension
\[
m_0:=|\Lambda_{\widetilde I_0}|
\asymp
2^{\widetilde J_0d}
\asymp
\varepsilon^{-\frac{d}{\min\{r_1,t\}-s}}.
\]
Therefore, in the dense direct least-squares model, forming the Gram matrix
already costs \(N_0m_0^2\).

It remains to identify the necessary sample size \(N_0\). Since the output
scale is fixed at \(j'=0\), the blockwise stochastic error at input scale \(j\)
has size
\[
\frac{2^{-jt+jr_1+jd/2}}{\sqrt{N_0}}.
\]
To make this no larger than \(\varepsilon\) uniformly over the retained input
scales \(0\le j\le J_0\), one needs
\[
N_0
\gtrsim
\varepsilon^{-2}
\sup_{0\le j\le J_0}
2^{2j(r_1-t)+dj}
\asymp
\varepsilon^{-2}
\varepsilon^{-\frac{(2(r_1-t)+d)_+}{t-s}}.
\]
Consequently, even this single-column regression problem forces the dense
least-squares cost
\[
N_0m_0^2
\gtrsim
\varepsilon^{
-\left(
2+
\frac{2d}{\min\{r_1,t\}-s}
+
\frac{(2(r_1-t)+d)_+}{t-s}
\right)
}.
\]
This is precisely the input-side exponent \(\kappa_{\rm in}\). The term
\(\min\{r_1,t\}-s\) reflects the cost of enlarging the regression support to
control omitted-variable bias; when \(r_1\ge t\), no enlargement is needed, but
when \(r_1<t\), the nuisance coordinates between \(I_0\) and
\(\widetilde I_0\) must be included in the regression even though they are not
retained in the final estimator.

The enlargement scale in the nested-support regression is not an artifact of the upper bound. To see this,
consider the following two-block population regression example inside the
single-column model. Let \(X_0\) denote one low-frequency input coordinate and
let \(X_L\) denote one input coordinate at scale \(L\). Suppose that
\[
\Var(X_0)\asymp 1,\qquad
\Var(X_L)\asymp 2^{-2Lr_1},
\qquad
\mathsf{Cov}(X_0,X_L)\asymp 2^{-Lr_1}.
\]
Equivalently, the preconditioned variables \(X_0\) and \(2^{Lr_1}X_L\) have a
nonzero correlation of order one. This covariance structure is compatible with
the diagonal preconditioning bounds in Lemma~\ref{lemma:property_UAW}.

Now consider the noiseless single-column regression model
\[
Y=X_0\beta+X_L\theta.
\]
If one regresses \(Y\) only on \(X_0\), omitting \(X_L\), the population least
squares coefficient is
\[
\beta_{\rm pop}
=
\beta+\frac{\mathsf{Cov}(X_0,X_L)}{\Var(X_0)}\theta.
\]
Hence the omitted-variable bias is
\[
|\beta_{\rm pop}-\beta|
\asymp
2^{-Lr_1}|\theta|.
\]
The operator-class constraint allows a coefficient at input scale \(L\) of size
\[
|\theta|\asymp 2^{Ls}.
\]
Therefore the omitted-variable bias can be as large as
\[
2^{-Lr_1}2^{Ls}
=
2^{-L(r_1-s)}.
\]
Since this bias appears in a low-frequency retained coordinate, the
\(H^t\)-loss weight does not reduce it. Thus, to make the omitted-variable bias
no larger than the target accuracy \(\varepsilon\), it is necessary that
\[
2^{-L(r_1-s)}\lesssim \varepsilon,
\]
or equivalently
\[
L\gtrsim \frac{\log_2(1/\varepsilon)}{r_1-s}.
\]
In the regime \(r_1<t\), this is precisely
\[
\widetilde J_0
\asymp
\frac{\log_2(1/\varepsilon)}{\min\{r_1,t\}-s}
=
\frac{\log_2(1/\varepsilon)}{r_1-s}.
\]
Thus the enlarged regression support has the optimal order: using a smaller
cutoff would allow high-frequency nuisance coefficients, already negligible in
the final \(H^t\)-loss, to leak into low-frequency retained coordinates and
produce an error larger than \(\varepsilon\).

\subsubsection{Output-side multi-response barrier}\label{subsubsec:output}
To isolate the output-side obstruction, fix a single low-frequency input index
\(\lambda_0\in\Lambda_0\), and fix an output scale \(j'\). Consider the
subclass of matrices satisfying
\[
(\bfA^\top)_{\mu,\lambda'}=0
\qquad\text{unless } \mu=\lambda_0
\quad\text{and}\quad \lambda'\in\Lambda_{j'}.
\]
Thus the only unknown coefficients form the row vector
\[
\bfy
:=
\bigl((\bfA^\top)_{\lambda_0,\lambda'}\bigr)_{\lambda'\in\Lambda_{j'}}
\in \mathbb R^{|\Lambda_{j'}|}.
\]
On this subclass, the block-column model reduces to a one-predictor,
multi-response regression problem
\[
\bfF_{j'}
=
\bfU_{:,\lambda_0}\bfy^\top+\bfW_{j'},
\]
where
\[
\bfU_{:,\lambda_0}\in\mathbb R^{N\times 1},
\qquad
\bfy\in\mathbb R^{|\Lambda_{j'}|},
\qquad \bfF_{j'},
\bfW_{j'}\in\mathbb R^{N\times |\Lambda_{j'}|}.
\]

For this restricted problem, the contribution to the weighted operator norm is
the Euclidean error of the row vector, with the output Sobolev weight:
\[
\|\bfD^{-t}(\widehat{\bfA}^\top-\bfA^\top)\bfD^{-t'}\|
=
2^{-j't'}\|\widehat{\bfy}-\bfy\|_2 .
\]
The noise at output scale \(j'\) has coordinate size \(2^{-j'r_2}\), and
\[
|\Lambda_{j'}|=|\nabla_{j'}|\asymp 2^{j'd}.
\]
Hence, if all coordinates of \(\bfy\) are estimated using \(N_{j'}\) samples,
the least-squares error has the scale
\[
\|\widehat{\bfy}-\bfy\|_2
\asymp
\frac{2^{-j'r_2}|\Lambda_{j'}|^{1/2}}{\sqrt{N_{j'}}}
\asymp
\frac{2^{-j'r_2+j'd/2}}{\sqrt{N_{j'}}}.
\]
Therefore, to make the weighted error at this output scale at most
\(\varepsilon\), one needs
\[
2^{-j't'}
\frac{2^{-j'r_2+j'd/2}}{\sqrt{N_{j'}}}
\lesssim
\varepsilon,
\]
or equivalently
\[
N_{j'}
\gtrsim
\varepsilon^{-2}2^{-2j't'-2j'r_2+j'd}.
\]
Since processing the responses at this scale costs at least
\(N_{j'}|\Lambda_{j'}|\), the cost at the fixed output scale \(j'\) is bounded
below by
\[
N_{j'}|\Lambda_{j'}|
\gtrsim
\varepsilon^{-2}2^{-2j't'-2j'r_2+2j'd}.
\]

This lower bound cannot be improved by assigning different sample sizes to
different response coordinates. Indeed, if \(N_{\lambda'}\) samples are used
for the coordinate \(\lambda'\in\Lambda_{j'}\), then the weighted squared error
constraint gives
\[
\sum_{\lambda'\in\Lambda_{j'}}
\frac{2^{-2j't'}2^{-2j'r_2}}{N_{\lambda'}}
\lesssim
\varepsilon^2,
\]
whereas the response-processing cost is at least
\[
\sum_{\lambda'\in\Lambda_{j'}}N_{\lambda'}.
\]
By Cauchy's inequality,
\[
\bigg(\sum_{\lambda'\in\Lambda_{j'}}N_{\lambda'}\bigg)
\bigg(\sum_{\lambda'\in\Lambda_{j'}}\frac1{N_{\lambda'}}\bigg)
\ge |\Lambda_{j'}|^2.
\]
Consequently,
\[
\sum_{\lambda'\in\Lambda_{j'}}N_{\lambda'}
\gtrsim
\varepsilon^{-2}2^{-2j't'-2j'r_2}|\Lambda_{j'}|^2
\asymp
\varepsilon^{-2}2^{-2j't'-2j'r_2+2j'd}.
\]
Thus the fixed-scale multi-response cost is optimal even under coordinatewise
sample allocation.

It remains to choose the largest output scale that is statistically relevant
at accuracy \(\varepsilon\). Since the row vector is admissible with
\[
\|\bfy\|_2\lesssim 2^{j's'},
\]
its weighted size is
\[
2^{-j't'}\|\bfy\|_2
\lesssim
2^{-j'(t'-s')}.
\]
Therefore, output scales satisfying
\[
2^{-j'(t'-s')}\lesssim \varepsilon
\]
can be discarded without affecting the target accuracy, while the finest
relevant scale satisfies
\[
j'\asymp \frac{\log_2(1/\varepsilon)}{t'-s'}.
\]

At this scale, the number of output coefficients is
\[
|\Lambda_{j'}|
\asymp
2^{j'd}
\asymp
\varepsilon^{-\frac{d}{t'-s'}},
\]
so explicit recovery already has the output-resolution floor
\[
\varepsilon^{-\frac{d}{t'-s'}}.
\]
Moreover, the fixed-scale multi-response lower bound becomes
\[
\varepsilon^{-2}2^{-2j't'-2j'r_2+2j'd}
\asymp
\varepsilon^{-\frac{2(-r_2-s')+2d}{t'-s'}}.
\]
Consequently, this restricted row-vector recovery problem already forces the
two output-side barriers
\[
\frac{d}{t'-s'}
\qquad\text{and}\qquad
\frac{2(-r_2-s')+2d}{t'-s'}.
\]
This explains the output-side exponent
\[
\kappa_{\rm out}
=
\max\left\{
\frac{d}{t'-s'},
\frac{2(-r_2-s')+2d}{t'-s'}
\right\}
\]
within the blockwise dense-regression framework.

\section*{Acknowledgments}
The authors were partly funded by the NSF CAREER award DMS-2237628. The authors also thank Wenwen Li for helpful discussions on the manuscript.

\bibliographystyle{siam} 
\bibliography{references}

\newpage

\appendix

\section{Auxiliary materials}\label{app:convergencerate}

This appendix is organized as follows. Appendices~\ref{app:bias},
\ref{app:ovb}, and~\ref{app:var} prove Lemmas~\ref{lemma:Error-Bias},
\ref{lemma:Error-OVB}, and~\ref{lemma:Error-Var}, respectively, which bound
the bias, omitted-variable bias, and variance terms used in the proof of the
upper bound in Proposition~\ref{prop:upper_bound}.
Appendix~\ref{app:truncated_estimator} introduces a stability-truncated version
of the estimator constructed in Subsection~\ref{subsec:minimax_upper} and
proves the in-expectation error bound claimed in Theorem~\ref{thm:main1},
without changing the computational cost exponent in Proposition~\ref{prop:cost}.
Appendix~\ref{app:diagonal_input} shows that, when the input covariance
\(\bfC_u\) is exactly diagonal, the standard estimator without nested-support
regression, namely the estimator with \( \widetilde{J}(j')\equiv J(j')\), also
achieves the minimax rate \(N^{-\gamma}\).

\subsection{Bias}\label{app:bias}

\begin{proof}[Proof of Lemma \ref{lemma:Error-Bias}]
Since $\mcA\in \mathcal{O}(s,s',R)$, the norm equivalence implies that $\bfA\in \mcM(s,s',R')$, namely,
\[
\|\bfD^{-s'}\bfA\bfD^{-s}\|\le R'.
\]
Recall that $\bfD^s$ is the diagonal matrix whose $j$-th block is $2^{js}I$, see \eqref{eq:D}. It follows that for each $j,j'\ge 0$,
\[
\bigl\|2^{-js'} \bfA_{j,j'} 2^{-j's}\bigr\|\le R',
\]
and hence
\[
\|\bfA_{j,j'}\|\le R' 2^{js'+j's}.
\]
Therefore,
\[
\|(\bfA^\top)_{j,j'}\|=\|\bfA_{j',j}\|\le R' 2^{js+j's'}.
\]

By the block Schur test,
\begin{align*}
\| \bfD^{-t} (\bfA^{\top})_{\mcE_N^c} \bfD^{-t'} \|
&\le
\bigg(\sup_j \sum_{j':(j,j')\in \mcE_N^c}
2^{-jt-j't'}\|(\bfA^\top)_{j,j'}\|\bigg)^{1/2} \\
&\qquad \times
\bigg(\sup_{j'} \sum_{j:(j,j')\in \mcE_N^c}
2^{-jt-j't'}\|(\bfA^\top)_{j,j'}\|\bigg)^{1/2} \\
&\le
R'
\bigg(\sup_j \sum_{j':(j,j')\in \mcE_N^c}
2^{-j(t-s)-j'(t'-s')}\bigg)^{1/2}
\bigg(\sup_{j'} \sum_{j:(j,j')\in \mcE_N^c}
2^{-j(t-s)-j'(t'-s')}\bigg)^{1/2}.
\end{align*}

Recall that
\[
\mcE_N
=
\Big\{
(\lambda,\lambda')\in \Lambda\times\Lambda:
j(t-s)+j'(t'-s')\le \gamma \log_2 N
\Big\}.
\]
Fix $j\ge 0$. If
\[
j\le \frac{\gamma\log_2 N}{t-s},
\]
then $(j,j')\in \mcE_N^c$ means
\[
j'>\frac{\gamma\log_2 N-j(t-s)}{t'-s'}.
\]
Hence,
\begin{align*}
\sum_{j':(j,j')\in \mcE_N^c} 2^{-j(t-s)-j'(t'-s')}
&=
2^{-j(t-s)}
\sum_{j'>\frac{\gamma\log_2 N-j(t-s)}{t'-s'}}
2^{-j'(t'-s')} \\
&\lesssim
2^{-j(t-s)}
\,2^{-\left(\frac{\gamma\log_2 N-j(t-s)}{t'-s'}\right)(t'-s')}
=
N^{-\gamma}.
\end{align*}
If instead
\[
j> \frac{\gamma\log_2 N}{t-s},
\]
then automatically $(j,j')\in \mcE_N^c$ for all $j'\ge 0$, so
\[
\sum_{j':(j,j')\in \mcE_N^c} 2^{-j(t-s)-j'(t'-s')}
=
2^{-j(t-s)}\sum_{j'\ge 0}2^{-j'(t'-s')}
\lesssim
2^{-j(t-s)}
\le N^{-\gamma}.
\]
Thus,
\[
\sup_j \sum_{j':(j,j')\in \mcE_N^c} 2^{-j(t-s)-j'(t'-s')}
\lesssim N^{-\gamma}.
\]
By the same argument,
\[
\sup_{j'} \sum_{j:(j,j')\in \mcE_N^c} 2^{-j(t-s)-j'(t'-s')}
\lesssim N^{-\gamma}.
\]
Substituting these bounds into the Schur estimate yields
\[
\ErrorBias
=
\| \bfD^{-t} (\bfA^{\top})_{\mcE_N^c} \bfD^{-t'} \|
\lesssim
N^{-\gamma}.
\]
\end{proof}

\subsection{Omitted-variable bias}\label{app:ovb}

\begin{proof}[Proof of Lemma \ref{lemma:Error-OVB}]

Recall that \(\bfB^{\mathrm{ovb}}\) is defined block-columnwise by
\[
(\bfB^{\mathrm{ovb}})_{\widetilde I_{j'},j'}:=\big((\bfU^{(j')}_{\widetilde I_{j'}})^{\top}
\bfU^{(j')}_{\widetilde I_{j'}}\big)^{-1}
(\bfU^{(j')}_{\widetilde I_{j'}})^{\top}
\bfU^{(j')}_{\widetilde I_{j'}^c}
(\bfA^\top)_{\widetilde I_{j'}^c,j'},
\]
for all
\[
j'\le \left\lceil \frac{\gamma \log_2 N}{t'-s'} \right\rceil,
\]
where
\[
\widetilde I_{j'}=\{0,1,\ldots,\widetilde J(j')\},
\qquad
\widetilde J(j')
=
\left\lceil
\frac{\gamma\log_2 N-j'(t'-s')}
{\min\{r_1,t\}-s}
\right\rceil .
\]

We first prove the following blockwise bound: with probability at least
\(1-\delta\), uniformly over all relevant \(j'\) and all \(j\in I_{j'}\),
\begin{align}\label{eq:ovb_block_bound_completed}
\|(\bfB^{\mathrm{ovb}})_{j,j'}\|
\lesssim
R' 2^{jr_1+j's'-\widetilde J(j')(r_1-s)} .
\end{align}

Write
\[
\widehat{\Sigma}:=\frac{1}{N_{j'}} (\bfU^{(j')})^\top \bfU^{(j')},
\qquad
\widehat{\bar\Sigma}:=\bfD^{r_1}\widehat{\Sigma}\bfD^{r_1},
\]
and let
\[
\Sigma:=\E \widehat{\Sigma},
\qquad
\bar\Sigma:=\bfD^{r_1}\Sigma\bfD^{r_1}.
\]
For notational simplicity, we suppress the dependence of these matrices on
\(j'\). However, throughout the analysis of the \(j'\)-th column block,
\(\widehat{\Sigma}\) is formed using only the corresponding \(N_{j'}\) samples.

By Lemma~\ref{lemma:property_UAW}, \(\bar\Sigma\) is uniformly
well-conditioned:
\[
c_- \le \sigma_{\min}(\bar\Sigma)
\le \sigma_{\max}(\bar\Sigma)\le c_+ .
\]
Moreover,
\[
(\bfB^{\mathrm{ovb}})_{\widetilde I_{j'},j'}
=
\bfD^{r_1}
(\widehat{\bar\Sigma}_{\widetilde I_{j'},\widetilde I_{j'}})^{-1}
\widehat{\bar\Sigma}_{\widetilde I_{j'},\widetilde I_{j'}^c}
\bfD^{-r_1}
(\bfA^\top)_{\widetilde I_{j'}^c,j'} .
\]
Equivalently,
\begin{equation}\label{eq:OVBexpression}
(\bfB^{\mathrm{ovb}})_{\widetilde I_{j'},j'}
=
\bfD^{r_1}
(\widehat{\bar\Sigma}_{\widetilde I_{j'},\widetilde I_{j'}})^{-1}
\widehat{\bar\Sigma}_{\widetilde I_{j'},\widetilde I_{j'}^c}
\bfD^{-(r_1-s)}
\bfD^{-s}
(\bfA^\top)_{\widetilde I_{j'}^c,j'} .
\end{equation}

We now control the two empirical factors. First, since
\[
|\Lambda_{\widetilde I_{j'}}|
\lesssim 2^{\widetilde J(j')d}\lesssim N_{j'},
\]
standard Gaussian sample covariance concentration gives that, with probability at
least \(1-\delta/2\), 
\[
\|(\widehat{\bar\Sigma}_{\widetilde I_{j'},\widetilde I_{j'}})^{-1}\|
\lesssim 1,
\]
provided that $N_{j'}\gtrsim 2^{\widetilde{J}(j')d} +\log(1/\delta)$. After a union bound
over the \(O(\log N)\) relevant values of \(j'\), we have that with probability at
least \(1-\delta/2\),
\begin{align}\label{eq:empirical_inverse_ovb}
\sup_{j'}
\|(\widehat{\bar\Sigma}_{\widetilde I_{j'},\widetilde I_{j'}})^{-1}\|
\lesssim 1 ,
\end{align}
provided \(N_{j'}\gtrsim 2^{\widetilde{J}(j')d} +\log((\log N)/\delta)\).

Second, we claim that, with probability at least \(1-\delta/2\),
\begin{align}\label{eq:empirical_cross_ovb}
\sup_{j'} 
\left\|
\widehat{\bar\Sigma}_{\widetilde I_{j'},\widetilde I_{j'}^c}
\bfD^{-(r_1-s)}
\right\|
\lesssim
2^{-\widetilde J(j')(r_1-s)} .
\end{align}
Indeed, decompose
\[
\widehat{\bar\Sigma}_{\widetilde I_{j'},\widetilde I_{j'}^c}
\bfD^{-(r_1-s)}
=
\bar\Sigma_{\widetilde I_{j'},\widetilde I_{j'}^c}
\bfD^{-(r_1-s)}
+
(\widehat{\bar\Sigma}-\bar\Sigma)_{\widetilde I_{j'},\widetilde I_{j'}^c}
\bfD^{-(r_1-s)} .
\]
The deterministic part is bounded by
\[
\left\|
\bar\Sigma_{\widetilde I_{j'},\widetilde I_{j'}^c}
\bfD^{-(r_1-s)}
\right\|
\le
\|\bar\Sigma\|\,
\|\bfD^{-(r_1-s)}_{\widetilde I_{j'}^c}\|
\lesssim
2^{-\widetilde J(j')(r_1-s)} .
\]

We next control the stochastic part by writing it explicitly as a sample
cross-covariance. For fixed \(j'\), define the coordinate projections
\[
P_{\widetilde I_{j'}}:\ell^2(\Lambda)\to
\ell^2(\Lambda_{\widetilde I_{j'}}),
\qquad
P_{\widetilde I_{j'}^c}:\ell^2(\Lambda)\to
\ell^2(\Lambda_{\widetilde I_{j'}^c}).
\]
Let
\[
X_i:=P_{\widetilde I_{j'}}\bfD^{r_1}\bfu_i,
\qquad
Y_i:=\bfD^{-(r_1-s)}
P_{\widetilde I_{j'}^c}\bfD^{r_1}\bfu_i .
\]
Then \(X_i\) and \(Y_i\) are centered jointly Gaussian random variables taking
values in
\(\ell^2(\Lambda_{\widetilde I_{j'}})\) and
\(\ell^2(\Lambda_{\widetilde I_{j'}^c})\), respectively. Moreover,
\begin{align*}
\frac{1}{N_{j'}}\sum_{i=1}^{N_{j'}} X_i Y^{\top}_i
&=
\frac{1}{N_{j'}}\sum_{i=1}^{N_{j'}}
\big(P_{\widetilde I_{j'}}\bfD^{r_1}\bfu_i\big)
\big(
\bfD^{-(r_1-s)}
P_{\widetilde I_{j'}^c}\bfD^{r_1}\bfu_i\big)^{\top} \\
&=
P_{\widetilde I_{j'}}
\bigg(
\frac{1}{N_{j'}}\sum_{i=1}^{N_{j'}}
(\bfD^{r_1}\bfu_i) (\bfD^{r_1}\bfu_i)^{\top}
\bigg)
P_{\widetilde I_{j'}^c}\bfD^{-(r_1-s)} \\
&=
\widehat{\bar\Sigma}_{\widetilde I_{j'},\widetilde I_{j'}^c}
\bfD^{-(r_1-s)} .
\end{align*}
Similarly,
\[
\E X Y^{\top}
=
\bar\Sigma_{\widetilde I_{j'},\widetilde I_{j'}^c}
\bfD^{-(r_1-s)} .
\]
Therefore,
\[
\frac{1}{N_{j'}}\sum_{i=1}^{N_{j'}} X_i Y^{\top}_i-\E X Y^{\top}
=
(\widehat{\bar\Sigma}-\bar\Sigma)_{\widetilde I_{j'},\widetilde I_{j'}^c}
\bfD^{-(r_1-s)} .
\]
We now estimate the covariance quantities in
Lemma~\ref{lemma:sample-cross-covariance}. Since
\(\bar\Sigma=\bfD^{r_1}\Sigma\bfD^{r_1}\) is uniformly bounded and bounded
below on \(\ell^2(\Lambda)\), the covariance operator of \(X\) is
\[
\Sigma_X
=
P_{\widetilde I_{j'}}\bar\Sigma P_{\widetilde I_{j'}} .
\]
Thus
\[
\|\Sigma_X\|\lesssim 1,
\qquad
\tr(\Sigma_X)
\lesssim
\operatorname{rank}(P_{\widetilde I_{j'}})
\lesssim
2^{\widetilde J(j')d}.
\]

Similarly, the covariance operator of \(Y\) is
\[
\Sigma_Y
=
\bfD^{-(r_1-s)}
P_{\widetilde I_{j'}^c}
\bar\Sigma
P_{\widetilde I_{j'}^c}
\bfD^{-(r_1-s)} .
\]
Since \(\|\bar\Sigma\|\lesssim 1\),
\[
\|\Sigma_Y\|
\lesssim
\|\bfD^{-(r_1-s)}P_{\widetilde I_{j'}^c}\|^2
\lesssim
2^{-2\widetilde J(j')(r_1-s)}.
\]
Moreover, using \(r_1-s>d/2\),
\begin{align*}
\tr(\Sigma_Y) \lesssim
\tr
\Big(
\bfD^{-2(r_1-s)}
P_{\widetilde I_{j'}^c}
\Big) \lesssim
\sum_{k>\widetilde J(j')}2^{kd}2^{-2k(r_1-s)} \lesssim
2^{-2\widetilde J(j')(r_1-s)}
2^{\widetilde J(j')d}.
\end{align*}

Applying Lemma~\ref{lemma:sample-cross-covariance} with
\(t\asymp \log((\log N)/\delta)\), and taking a union bound over the
\(O(\log N)\) possible values of \(j'\), we obtain, with probability at least
\(1-\delta\), uniformly over all relevant \(j'\),
\begin{align*}
\left\|
(\widehat{\bar\Sigma}-\bar\Sigma)_{\widetilde I_{j'},\widetilde I_{j'}^c}
\bfD^{-(r_1-s)}
\right\|
&\lesssim
2^{-\widetilde J(j')(r_1-s)}
\left(
\sqrt{\frac{2^{\widetilde J(j')d}+\log((\log N)/\delta)}{N_{j'}}}
+
\frac{2^{\widetilde J(j')d}+\log((\log N)/\delta)}{N_{j'}}
\right).
\end{align*}
In particular, since \(N_{j'}\gtrsim 2^{\widetilde J(j')d}+ \log((\log N)/\delta)\), we have
\[
\left\|
(\widehat{\bar\Sigma}-\bar\Sigma)_{\widetilde I_{j'},\widetilde I_{j'}^c}
\bfD^{-(r_1-s)}
\right\|
\lesssim
2^{-\widetilde J(j')(r_1-s)}.
\]

Finally, since
\begin{equation}\label{eq:auxiliarybound}
  \|\bfD^{-s}(\bfA^\top)_{\widetilde I_{j'}^c,j'}\|
\le
\|\bfD^{-s}\bfA^\top\bfD^{-s'}\|\,2^{j's'}
\lesssim
R'2^{j's'},  
\end{equation}
we combine the expression for $(\bfB^{\mathrm{ovb}})_{j,j'}$ in \eqref{eq:OVBexpression}  with the bounds in \eqref{eq:empirical_inverse_ovb}, 
\eqref{eq:empirical_cross_ovb} and \eqref{eq:auxiliarybound} to get, for every \(j\in \widetilde I_{j'}\),
\[
\|(\bfB^{\mathrm{ovb}})_{j,j'}\|
\lesssim
2^{jr_1}
\cdot 1
\cdot
2^{-\widetilde J(j')(r_1-s)}
\cdot
R'2^{j's'} ,
\]
which proves \eqref{eq:ovb_block_bound_completed}.

We now pass from the blockwise estimate to the weighted operator norm. Since
\((\bfB^{\mathrm{ovb}})_{\mcE_N}\) only keeps the blocks
\((j,j')\in\mcE_N\), the block Schur test gives
\begin{align*}
\ErrorOVB
&=
\|\bfD^{-t}(\bfB^{\mathrm{ovb}})_{\mcE_N}\bfD^{-t'}\| \\
&\le
\bigg(
\sup_j
\sum_{j':(j,j')\in \mcE_N}
2^{-jt-j't'}\|(\bfB^{\mathrm{ovb}})_{j,j'}\|
\bigg)^{1/2} \\
&\qquad\times
\bigg(
\sup_{j'}
\sum_{j:(j,j')\in \mcE_N}
2^{-jt-j't'}\|(\bfB^{\mathrm{ovb}})_{j,j'}\|
\bigg)^{1/2} \\
&\lesssim
R'
\bigg(
\sup_j
\sum_{j':(j,j')\in \mcE_N}
2^{-j(t-r_1)-j'(t'-s')}
2^{-\widetilde J(j')(r_1-s)}
\bigg)^{1/2} \\
&\qquad\times
\bigg(
\sup_{j'}
\sum_{j:(j,j')\in \mcE_N}
2^{-j(t-r_1)-j'(t'-s')}
2^{-\widetilde J(j')(r_1-s)}
\bigg)^{1/2}.
\end{align*}

It remains to bound the two Schur sums. We claim that, for every
\((j,j')\in\mcE_N\),
\[
2^{-jt-j't'}2^{jr_1+j's'-\widetilde J(j')(r_1-s)}
\lesssim N^{-\gamma}.
\]
Indeed, if \(r_1<t\), then
\[
\widetilde J(j')
\ge
\frac{\gamma\log_2 N-j'(t'-s')}{r_1-s},
\]
and hence
\[
2^{-jt-j't'}2^{jr_1+j's'-\widetilde J(j')(r_1-s)}
\le
N^{-\gamma}2^{-j(t-r_1)}
\le N^{-\gamma}.
\]
If \(r_1\ge t\), then
\[
\widetilde J(j')
\ge
\frac{\gamma\log_2 N-j'(t'-s')}{t-s}.
\]
Since \((j,j')\in\mcE_N\), we have
\[
j(t-s)\le \gamma\log_2 N-j'(t'-s').
\]
Therefore,
\begin{align*}
2^{-jt-j't'}2^{jr_1+j's'-\widetilde J(j')(r_1-s)} &\le
2^{j(r_1-t)-j'(t'-s')}
2^{-\frac{r_1-s}{t-s}(\gamma\log_2 N-j'(t'-s'))} \\
&\le
2^{\frac{r_1-t}{t-s}(\gamma\log_2 N-j'(t'-s'))-j'(t'-s')}
2^{-\frac{r_1-s}{t-s}(\gamma\log_2 N-j'(t'-s'))} \\
&=
N^{-\gamma}.
\end{align*}
This proves the claim.

Since the number of admissible scale blocks in each row and column of
\(\mcE_N\) is \(O(\log N)\), both Schur sums are bounded by
\(N^{-\gamma}\log N\). Consequently, for any $\delta\in(0,1)$, if for all relevant $j'$ it holds that
\[
N_{j'}\gtrsim 2^{\widetilde J(j')d}
+\log ((\log N)/\delta),
\]
 then, with probability at least $1-\delta,$ 
\[
\ErrorOVB\lesssim N^{-\gamma}\log N .
\]
\end{proof}

The proof of Lemma~\ref{lemma:Error-OVB} uses the following
dimension-free concentration inequality for sample cross-covariance matrices
\cite{chen2026cross}.
\begin{lemma}[{\cite[Theorem 2.1]{chen2026cross}}]
\label{lemma:sample-cross-covariance}

Let $(X,Y),(X_1,Y_1),\ldots,
(X_N,Y_N)$ be i.i.d. centered jointly
Gaussian random vectors. Let \(\Sigma_X\) and \(\Sigma_Y\) be the covariances of \(X\) and
\(Y\), respectively, and let
\[
r_X:=\frac{\tr(\Sigma_X)}{\|\Sigma_X\|},
\qquad
r_Y:=\frac{\tr(\Sigma_Y)}{\|\Sigma_Y\|}
\]
denote their effective ranks. Then, for every \(t\ge 1\), with probability at least \(1-e^{-t}\),
\begin{align*}
\bigg\|
\frac{1}{N}\sum_{i=1}^{N}X_i Y_i^{\top}-\E X Y^{\top}
\bigg\|
\lesssim
\bigl(\|\Sigma_X\|\|\Sigma_Y\|\bigr)^{1/2}
\bigg(
\sqrt{\frac{r_X+r_Y+t}{N}}
+
\frac{\sqrt{(r_X+t)(r_Y+t)}}{N}
\bigg).
\end{align*}
\end{lemma}

\subsection{Variance}\label{app:var}

\begin{proof}[Proof of Lemma \ref{lemma:Error-Var}]
Recall that \(\bfB^{\mathrm{var}}\) is defined block-columnwise by
\[
(\bfB^{\mathrm{var}})_{\widetilde I_{j'},j'}:=\big((\bfU^{(j')}_{\widetilde I_{j'}})^{\top}
\bfU^{(j')}_{\widetilde I_{j'}}\big)^{-1}
(\bfU^{(j')}_{\widetilde I_{j'}})^{\top}
\bfW^{(j')}_{j'}
\]
for all
\[
j'\le
\left\lceil
\frac{\gamma\log_2 N}{t'-s'}
\right\rceil ,
\]
where
\[
\widetilde I_{j'}
=
\{0,1,\ldots,\widetilde J(j')\},
\qquad
\widetilde J(j')
=
\left\lceil
\frac{\gamma\log_2 N-j'(t'-s')}
{\min\{r_1,t\}-s}
\right\rceil .
\]
Since \((\bfB^{\mathrm{var}})_{\mcE_N}\) only retains the blocks
\((j,j')\in\mcE_N\), it suffices to consider \(j\in I_{j'}\subseteq
\widetilde I_{j'}\).

We first establish a uniform blockwise bound. For fixed \(j'\), define the
coordinate projection
\[
P_j:\ell^2(\Lambda_{\widetilde I_{j'}})\to \ell^2(\nabla_j),
\qquad j\in \widetilde I_{j'}.
\]
Then
\[
(\bfB^{\mathrm{var}})_{j,j'}
=
P_j
\big((\bfU^{(j')}_{\widetilde I_{j'}})^{\top}
\bfU^{(j')}_{\widetilde I_{j'}}\big)^{-1}
(\bfU^{(j')}_{\widetilde I_{j'}})^{\top}
\bfW^{(j')}_{j'} .
\]

Let \(\mcG_N\) be the event on which, uniformly over all relevant \(j'\),
\begin{equation}\label{eq:good-design-var}
cN_{j'}\bfD_{\widetilde I_{j'}}^{-2r_1}
\preceq
(\bfU^{(j')}_{\widetilde I_{j'}})^\top
\bfU^{(j')}_{\widetilde I_{j'}}
\preceq
CN_{j'}\bfD_{\widetilde I_{j'}}^{-2r_1}.
\end{equation}
We denote its complement by \(\mcG_N^c\).
As in the proof of Lemma~\ref{lemma:Error-OVB}, the definition of
\(\gamma\) implies
\[
|\Lambda_{\widetilde I_{j'}}|
\lesssim
2^{\widetilde J(j')d}
\lesssim N_{j'}
\]
uniformly over all relevant \(j'\). Therefore, by Gaussian sample covariance
concentration for the preconditioned design matrix and a union bound over the
\(O(\log N)\) possible values of \(j'\), we have
\[
\mathbb P(\mcG_N)\ge 1-\delta/2,
\]
provided \(N_{j'}\gtrsim 2^{\widetilde{J}(j')d}+\log((\log N)/\delta)\) for every relevant $j'$. 

On \(\mcG_N\), 
\[
\big((\bfU^{(j')}_{\widetilde I_{j'}})^{\top}
\bfU^{(j')}_{\widetilde I_{j'}}\big)^{-1}
\preceq
\frac{1}{c N_{j'}}\bfD_{\widetilde I_{j'}}^{2r_1}.
\]
Hence, using \(\|M\|^2=\|MM^\top\|\), we obtain
\begin{align*}
\left\|
P_j
\big((\bfU^{(j')}_{\widetilde I_{j'}})^{\top}
\bfU^{(j')}_{\widetilde I_{j'}}\big)^{-1}
(\bfU^{(j')}_{\widetilde I_{j'}})^{\top}
\right\|^2
&= \left\|
P_j
\big((\bfU^{(j')}_{\widetilde I_{j'}})^{\top}
\bfU^{(j')}_{\widetilde I_{j'}}\big)^{-1}
(\bfU^{(j')}_{\widetilde I_{j'}})^{\top} \bfU^{(j')}_{\widetilde I_{j'}} ((\bfU^{(j')}_{\widetilde I_{j'}})^\top\bfU^{(j')}_{\widetilde I_{j'}})^{-1} P_j^{\top}
\right\| \nonumber\\
&=
\left\|
P_j
\big((\bfU^{(j')}_{\widetilde I_{j'}})^{\top}
\bfU^{(j')}_{\widetilde I_{j'}}\big)^{-1}
P_j^\top
\right\| \lesssim
\frac{2^{2jr_1}}{N_{j'}},
\end{align*}
and similarly
\begin{align*}
\left\|
P_j
\big((\bfU^{(j')}_{\widetilde I_{j'}})^{\top}
\bfU^{(j')}_{\widetilde I_{j'}}\big)^{-1}
(\bfU^{(j')}_{\widetilde I_{j'}})^{\top}
\right\|_{\mathrm{F}}^2
=
\tr\left(
P_j
\big((\bfU^{(j')}_{\widetilde I_{j'}})^{\top}
\bfU^{(j')}_{\widetilde I_{j'}}\big)^{-1}
P_j^\top
\right) \lesssim
\frac{2^{2jr_1}|\nabla_j|}{N_{j'}}
\asymp
\frac{2^{2jr_1+jd}}{N_{j'}}.
\end{align*}
Equivalently,
\begin{equation}\label{eq:var-design-norms}
\left\|
P_j
\big((\bfU^{(j')}_{\widetilde I_{j'}})^{\top}
\bfU^{(j')}_{\widetilde I_{j'}}\big)^{-1}
(\bfU^{(j')}_{\widetilde I_{j'}})^{\top}
\right\|
\lesssim
\frac{2^{jr_1}}{\sqrt N_{j'}},
\qquad
\left\|
P_j
\big((\bfU^{(j')}_{\widetilde I_{j'}})^{\top}
\bfU^{(j')}_{\widetilde I_{j'}}\big)^{-1}
(\bfU^{(j')}_{\widetilde I_{j'}})^{\top}
\right\|_{\mathrm{F}}
\lesssim
\frac{2^{jr_1+\frac{jd}{2}}}{\sqrt N_{j'}}.
\end{equation}

We now condition on \(\bfU\). Since the input samples and the noise samples are
independent, the matrix
\[
P_j
\big((\bfU^{(j')}_{\widetilde I_{j'}})^{\top}
\bfU^{(j')}_{\widetilde I_{j'}}\big)^{-1}
(\bfU^{(j')}_{\widetilde I_{j'}})^{\top}
\]
is deterministic under this conditioning. Moreover, \(\bfW^{(j')}_{j'}\) is an
\(N_{j'}\times |\nabla_{j'}|\) Gaussian matrix with independent rows distributed as
\[
(\bfw_i)_{j'}^\top\sim \mcN(0,\Sigma_{j'}),
\qquad
\Sigma_{j'}:=(\bfC_w)_{j',j'}.
\]
By Lemma~\ref{lemma:property_UAW} \ref{diagonal_preconditioning} \ref{preconditioning_noise},
\[
\|\Sigma_{j'}\|
\lesssim
2^{-2j'r_2},
\qquad
\tr(\Sigma_{j'})
\lesssim
2^{j'd}2^{-2j'r_2}.
\]
Equivalently, we may write
\[
\bfW^{(j')}_{j'}=\bfZ^{(j')}_{j'}\Sigma_{j'}^{1/2},
\]
where \(\bfZ_{j'}\in\mathbb R^{N_{j'}\times |\nabla_{j'}|}\) has i.i.d.
\(\mcN(0,1)\) entries and is independent of \(\bfU\), and
\begin{align}\label{eq:var-noise-norms}
\|\Sigma_{j'}^{1/2}\|
\lesssim
2^{-j'r_2},
\qquad
\|\Sigma_{j'}^{1/2}\|_{\mathrm{F}}
\lesssim
2^{-j'r_2+\frac{j'd}{2}}.
\end{align}

Applying Chevet's inequality (Lemma~\ref{lemma:chevet}) conditionally on
\(\bfU\), and using \eqref{eq:var-design-norms} and \eqref{eq:var-noise-norms}, gives that, with conditional
probability at least \(1-\eta\),
\begin{align}\label{eq:var-block-before-union}
\|(\bfB^{\mathrm{var}})_{j,j'}\|
&=
\left\|
P_j
\big((\bfU^{(j')}_{\widetilde I_{j'}})^{\top}
\bfU^{(j')}_{\widetilde I_{j'}}\big)^{-1}
(\bfU^{(j')}_{\widetilde I_{j'}})^{\top}
\bfZ^{(j')}_{j'}\Sigma_{j'}^{1/2}
\right\| \nonumber\\
&\lesssim
\frac{2^{jr_1+\frac{jd}{2}}}{\sqrt{N_{j'}}}2^{-j'r_2}
+
\frac{2^{jr_1}}{\sqrt{N_{j'}}}
2^{-j'r_2+\frac{j'd}{2}}
+
\frac{2^{jr_1}}{\sqrt{N_{j'}}}2^{-j'r_2}
\sqrt{\log(1/\eta)} \nonumber\\
&\asymp
\frac{
2^{jr_1-j'r_2+\frac d2\max\{j,j'\}}
}{\sqrt{N_{j'}}}
+
\frac{2^{jr_1-j'r_2}}{\sqrt{N_{j'}}}\sqrt{\log(1/\eta)} \nonumber\\
&\lesssim \sqrt{\log(1/\eta)}
\frac{
2^{jr_1-j'r_2+\frac d2\max\{j,j'\}}
}{\sqrt{N_{j'}}}.
\end{align}

Fix a sufficiently large constant \(C_1>0\), and define the event
\(\mcH_N=\mcH_N(\bfU,\bfW)\) by
\[
\mcH_N
:=
\bigcap_{(j,j')\in\mcE_N}
\left\{
\|(\bfB^{\mathrm{var}})_{j,j'}\|
\le
C_1\sqrt{\log(1/\eta)}
\frac{
2^{jr_1-j'r_2+\frac d2\max\{j,j'\}}
}{\sqrt{N_{j'}}}
\right\}.
\]
We write \(\mcH_N^c=\mcH_N^c(\bfU,\bfW)\) for its complement. Applying a union bound over all \((j,j')\in\mcE_N\) to the conditional estimate \eqref{eq:var-block-before-union}, we obtain that, for every fixed realization of \(\bfU\) for which \(\mcG_N\) holds,
\[
\mathbb P_{\bfW}\big(\mcH_N^c(\bfU,\bfW)\big)
\le
C_0(\log N)^2\eta,
\]
where \(\mathbb P_{\bfW}\) denotes probability only over the noise variables. Hence, by integrating over the realizations of
\(\bfU\),
\[
\begin{aligned}
\mathbb{P}(\mcG_N\cap\mcH_N^c)
=
\int_{\mcG_N}
\mathbb{P}_{\bfW}\big(\mcH^c_N(\bfU,\bfW)\big)
\,d\mathbb{P}_{\bfU}(\bfU) \le
C_0(\log N)^2\eta.
\end{aligned}
\]
Choosing $\eta=\frac{\delta}{2C_0(\log N)^2}$, we obtain
\[
\mathbb{P}(\mcG_N\cap\mcH_N^c)\le \frac{\delta}{2}.
\]
Together with \(\mathbb{P}(\mcG_N^c)\le \delta/2\), this gives
\begin{align*}
\mathbb{P}(\mcG_N\cap\mcH_N)
&=
1-\mathbb{P}\big((\mcG_N\cap\mcH_N)^c\big) \\
&=
1-\mathbb{P}\big(\mcG_N^c\cup(\mcG_N\cap\mcH_N^c)\big) \\
&\ge
1-\mathbb{P}(\mcG_N^c)
-\mathbb{P}(\mcG_N\cap\mcH_N^c) \\
&\ge
1-\frac{\delta}{2}-\frac{\delta}{2}
=
1-\delta .
\end{align*}
Since \(\log(1/\eta)\lesssim \log(N/\delta)\), on the event \(\mcG_N\cap\mcH_N\), which has probability at least \(1-\delta\), uniformly over all \((j,j')\in\mcE_N\),
\begin{equation}\label{eq:var-block-bound}
\|(\bfB^{\mathrm{var}})_{j,j'}\|
\lesssim
\sqrt{\log\Big(\frac{N}{\delta}\Big)}
\frac{
2^{jr_1-j'r_2+\frac d2\max\{j,j'\}}
}{\sqrt{N_{j'}}}
\lesssim
\sqrt{\log\Big(\frac{N}{\delta}\Big)} N^{-\gamma} 2^{jt+j't'},
\end{equation}
where the last inequality follows from the definition of \(N_{j'}\) in
\eqref{eq:Njprime_box}:
\[
N_{j'}\gtrsim N^{2\gamma}\sup_{0\le j\le J(j')}
2^{-2jt-2j't'+2jr_1-2j'r_2+d\max\{j,j'\}}.
\]

We now pass from the blockwise estimate to the weighted operator norm. By the
block Schur test,
\begin{align*}
\ErrorVar
&=
\|\bfD^{-t}(\bfB^{\mathrm{var}})_{\mcE_N}\bfD^{-t'}\| \\
&\le
\bigg(
\sup_j
\sum_{j':(j,j')\in\mcE_N}
2^{-jt-j't'}
\|(\bfB^{\mathrm{var}})_{j,j'}\|
\bigg)^{1/2} \\
&\qquad\times
\bigg(
\sup_{j'}
\sum_{j:(j,j')\in\mcE_N}
2^{-jt-j't'}
\|(\bfB^{\mathrm{var}})_{j,j'}\|
\bigg)^{1/2}\\
&\lesssim
N^{-\gamma}
\sqrt{\log\!\Big(\frac{N}{\delta}\Big)}
\log N.
\end{align*}
This completes the proof.
\end{proof}

The following form of Chevet's inequality provides sharp control of the
operator norm of weighted Gaussian random matrices. The expectation bound is a
direct consequence of Chevet's inequality
\cite[Theorem 8.7.1]{vershynin2018high}, while the high-probability estimate
follows from the standard Gaussian concentration inequality for Lipschitz
functions. We omit the proof.

\begin{lemma}[{\cite[Theorem 8.7.1]{vershynin2018high}}]\label{lemma:chevet}
Let \(G\in \mathbb R^{m\times n}\) have i.i.d. \(\mcN(0,1)\) entries, and let
\(A\in\mathbb R^{p\times m}\), \(B\in\mathbb R^{n\times q}\) be deterministic
matrices. Then
\[
\E \|A G B\|
\asymp
\|A\|_{\mathrm{F}} \|B\|+\|A\|\|B\|_{\mathrm{F}}.
\]
Moreover, for every \(\delta\in(0,1)\), with probability at least
\(1-\delta\),
\[
\|A G B\|
\lesssim
\|A\|_{\mathrm{F}} \|B\|+\|A\|\|B\|_{\mathrm{F}}
+
\|A\|\|B\|\sqrt{\log(1/\delta)} .
\]
\end{lemma}

\subsection{Truncated estimator and expectation bound}
\label{app:truncated_estimator}

In this subsection we show that a stability-truncated version of the
estimator constructed in Subsection~\ref{subsec:minimax_upper} satisfies an
in-expectation error bound. The truncation is only used to control the
contribution of rare ill-conditioned or large-data events. It does not change
the statistical rate, nor the computational cost exponent.

We fix a constant \(q>0\), to be chosen sufficiently large below, and set
\[
\delta_N:=N^{-q}.
\]
Throughout this subsection, the columnwise sample sizes \(N_{j'}\) are those
defined in \eqref{eq:Njprime_box} with \(\delta=\delta_N\). This choice only
changes \(N_{j'}\) by logarithmic factors compared with any fixed polynomial
confidence level.

Let
\[
J'_{\max}:=
\left\lceil \frac{\gamma\log_2 N}{t'-s'}\right\rceil .
\]
For each \(0\le j'\le J'_{\max}\), write
\[
m_{j'}:=|\Lambda_{\widetilde I_{j'}}|\asymp 2^{\widetilde{J}(j') d},
\qquad
n_{j'}:=|\Lambda_{j'}|=|\nabla_{j'}|\asymp 2^{j'd}.
\]
Recall that \(\bfU^{(j')}_{\widetilde I_{j'}}\) and
\(\bfF^{(j')}_{j'}\) denote the submatrices formed by selecting the same
\(N_{j'}\) sample rows. Define the Gram matrix
\[
\bfG_{j'}:=
(\bfU^{(j')}_{\widetilde I_{j'}})^\top
\bfU^{(j')}_{\widetilde I_{j'}}.
\]
For a sufficiently small fixed constant \(c_0>0\), define the stability event
\[
\mcS_{j'}:=
\left\{
\bfG_{j'}\succeq
c_0 N_{j'}\bfD_{\widetilde I_{j'}}^{-2r_1}
\right\}.
\]
We also introduce the data-size cutoff
\[
\mcB_N:=
\bigcap_{0\le j'\le J'_{\max}}
\left\{
\|\bfU^{(j')}_{\widetilde I_{j'}}\|_{\mathrm{F}}\le N^{K_0},
\quad
\|\bfF^{(j')}_{j'}\|_{\mathrm{F}}\le N^{K_0}
\right\},
\]
where \(K_0>0\) is chosen sufficiently large, depending only on the fixed
regularity parameters and on \(R\).

We now define the truncated coefficient estimator block-columnwise. For each
\(0\le j'\le J'_{\max}\), set
\[
((\widehat{\bfA}^{\top})^{\rm tr})_{\widetilde I_{j'},j'}
:=
\begin{cases}
\bfG_{j'}^{-1}
(\bfU^{(j')}_{\widetilde I_{j'}})^\top
\bfF^{(j')}_{j'},
&
\text{on } \mcS_{j'}\cap\mcB_N,
\\[0.4em]
0,
&
\text{on } (\mcS_{j'}\cap\mcB_N)^c.
\end{cases}
\]
All remaining entries of \((\widehat{\bfA}^{\top})^{\rm tr}\) are set to zero.
The operator estimator \(\widehat{\mcA}^{\rm tr}\) is obtained by the same final
restriction to the estimation region \(\mcE_N\):
\[
\widehat{\mcA}^{\rm tr}
:=
\sum_{(\lambda,\lambda')\in\mcE_N}
((\widehat{\bfA}^{\top})^{\rm tr})_{\lambda,\lambda'}\,
\varphi_{\lambda'}\otimes \varphi_{\lambda}.
\]

\begin{lemma}[In-expectation bound for the truncated estimator]
\label{lemma:truncated_expectation}
Under the setting of Theorem~\ref{thm:main1}, if \(A>0\) is chosen sufficiently
large, then
\[
\E
\|\widehat{\mcA}^{\rm tr}-\mcA\|_{H^t\to H^{-t'}}
\lesssim
N^{-\gamma}(\log N)^{3/2}.
\]
\end{lemma}

\begin{proof}
%[Proof of Lemma \ref{lemma:truncated_expectation}]
Let \(\mcG_N\) denote the intersection of the good events used in the proof of
Proposition~\ref{prop:upper_bound}, with confidence parameter
\(\delta_N=N^{-q}\). In particular, on \(\mcG_N\), the Gram lower bounds
\(\mcS_{j'}\) hold for every \(0\le j'\le J'_{\max}\), and the estimator
\(\widehat{\mcA}\) constructed with the sample sizes corresponding to
\(\delta_N\) satisfies
\[
\|\widehat{\mcA}-\mcA\|_{H^t\to H^{-t'}}
\le
C
N^{-\gamma}
\sqrt{\log\left(\frac{N}{\delta_N}\right)}
\log N .
\]
By Proposition~\ref{prop:upper_bound}, after enlarging constants if necessary,
\[
\Prob(\mcG_N^c)\le \delta_N=N^{-q}.
\]
Since \(\log(N/\delta_N)=(q+1)\log N\), this gives
\begin{equation}
\label{eq:truncated_good_event_bound}
\|\widehat{\mcA}-\mcA\|_{H^t\to H^{-t'}}
\le
C_q N^{-\gamma}(\log N)^{3/2}
\qquad\text{on } \mcG_N .
\end{equation}

We next control the data-size cutoff. The matrices appearing in the definition
of \(\mcB_N\) have polynomially many entries in \(N\), and their entries are
jointly Gaussian with variances bounded by a polynomial in \(N\). Therefore, by standard Gaussian tail bounds and a
union bound over \(0\le j'\le J'_{\max}=O(\log N)\), choosing \(K_0\)
sufficiently large gives
\begin{equation}
\label{eq:BN_tail}
\Prob(\mcB_N^c)
\le
e^{-cN}.
\end{equation}
On the event \(\mcG_N\cap\mcB_N\), all stability cutoffs are active and all
data-size cutoffs are satisfied. Hence
\[
\widehat{\mcA}^{\rm tr}=\widehat{\mcA}
\qquad
\text{on } \mcG_N\cap\mcB_N.
\]
Combining this identity with \eqref{eq:truncated_good_event_bound}, we obtain 
\[
\|\widehat{\mcA}^{\rm tr}-\mcA\|_{H^t\to H^{-t'}}
\le
C_q N^{-\gamma}(\log N)^{3/2}
\qquad
\text{on } \mcG_N\cap\mcB_N.
\]

It remains to bound the contribution of the complementary event. We claim that
there exists \(K>0\), depending only on the fixed parameters, such that
\begin{equation}
\label{eq:truncated_deterministic_bound}
\|\widehat{\mcA}^{\rm tr}-\mcA\|_{H^t\to H^{-t'}}
\le
N^K
\qquad
\text{deterministically}.
\end{equation}
Indeed, on \(\mcS_{j'}\cap\mcB_N\),
\[
\|((\widehat{\bfA}^{\top})^{\rm tr})_{\widetilde I_{j'},j'}\|
\le
\|\bfG_{j'}^{-1}\|\,
\|\bfU^{(j')}_{\widetilde I_{j'}}\|\,
\|\bfF^{(j')}_{j'}\|.
\]
The stability event implies
\[
\|\bfG_{j'}^{-1}\|
\le
C N_{j'}^{-1}
\|\bfD_{\widetilde I_{j'}}^{2r_1}\|
\le C 2^{2r_1\widetilde{J}(j')} \le 
C N^{K_1},
\]
for some constant $K_1>0$, because \(N_{j'}\ge 1\) and all scales in \(\widetilde I_{j'}\) are
\(O(\log N)\). On \(\mcB_N\),
\[
\|\bfU^{(j')}_{\widetilde I_{j'}}\|\le
\|\bfU^{(j')}_{\widetilde I_{j'}}\|_{\mathrm{F}}\le N^{K_0},
\qquad
\|\bfF^{(j')}_{j'}\|\le
\|\bfF^{(j')}_{j'}\|_{\mathrm{F}}\le N^{K_0}.
\]
Thus every retained block of
\((\widehat{\bfA}^{\top})^{\rm tr}\) has operator norm at most a polynomial in
\(N\). If \(\mcS_{j'}\cap\mcB_N\) fails, the corresponding block is set to
zero. Since the number of retained scale blocks and the Sobolev weights on
those blocks are polynomial in \(N\), it follows that
\[
\|\widehat{\mcA}^{\rm tr}\|_{H^t\to H^{-t'}}\le N^K.
\]
Moreover, because \(\mcA\in\mcO(s,s',R)\) and \(t>s\), \(t'>s'\),
\[
\|\mcA\|_{H^t\to H^{-t'}}\lesssim R.
\]
This proves \eqref{eq:truncated_deterministic_bound}, after increasing \(K\)
if necessary.

We now combine the preceding estimates:
\begin{align*}
\E
\|\widehat{\mcA}^{\rm tr}-\mcA\|_{H^t\to H^{-t'}}
&\le
C_q N^{-\gamma}(\log N)^{3/2}
+
N^K\Prob\big((\mcG_N\cap\mcB_N)^c\big) \\
&\le
C_q N^{-\gamma}(\log N)^{3/2}
+
N^K\bigl(N^{-q}+e^{-cN}\bigr).
\end{align*}
Choosing \(q>K+\gamma+1\), the second term is absorbed into the first, and
therefore
\[
\E
\|\widehat{\mcA}^{\rm tr}-\mcA\|_{H^t\to H^{-t'}}
\lesssim
N^{-\gamma}(\log N)^{3/2}.
\]
This completes the proof.
\end{proof}

We finally note that the stability truncation does not change the computational
cost exponent. The sample sizes \(N_{j'}\) are defined with
\(\delta_N=N^{-q}\), so the confidence term
\[
\log\left(\frac{\log N}{\delta_N}\right)
=
\log\bigl(N^q\log N\bigr)
\]
is still only logarithmic in \(N\). Therefore the estimates in
Proposition~\ref{prop:cost} are unchanged up to polylogarithmic factors.

Moreover, for each \(j'\), the Gram matrix \(\bfG_{j'}\) is already formed in
the dense least-squares computation. Checking the stability cutoff
\(\mcS_{j'}\) by Cholesky factorization or eigendecomposition costs
\(O(m_{j'}^3)\), where
\[
m_{j'}=|\Lambda_{\widetilde I_{j'}}|\asymp 2^{\widetilde J(j')d}.
\]
Since \(N_{j'}\gtrsim m_{j'}\), this is bounded by
\(O(N_{j'}m_{j'}^2)\), the same order as the Gram-matrix computation. The
data-size cutoff requires computing Frobenius norms of the matrices already
used in the regressions, with total cost
\[
\mcO\left(
\sum_{j'=0}^{J'_{\max}}
N_{j'}(m_{j'}+n_{j'})
\right),
\]
which is dominated by
\[
\mcO\left(
\sum_{j'=0}^{J'_{\max}}
N_{j'}(m_{j'}^2+m_{j'}n_{j'})
\right),
\]
the dense least-squares cost analyzed in Proposition~\ref{prop:cost}. Hence
\(\widehat{\mcA}^{\rm tr}\) has the same computational cost as
\(\widehat{\mcA}\), up to polylogarithmic factors.

\subsection{Diagonal input covariance}\label{app:diagonal_input}

In this subsection, we show that, when \(\bfC_u\) is exactly diagonal, the standard
estimator without nested-support regression, namely the estimator with
\(J(j')\equiv \widetilde J(j')\), also achieves the minimax rate
\(N^{-\gamma}\). The bias and variance analyses are unchanged; it therefore
remains only to revisit the omitted-variable bias term.

Since \(\bfC_u\) is diagonal, the population cross-covariance between disjoint
coordinate blocks vanishes. Therefore the population omitted-variable bias is zero, and only
the empirical cross-covariance fluctuation needs to be controlled.

We first prove the following blockwise bound: with probability at least
\(1-\delta\), uniformly over all \((j,j')\in\mcE_N\),
\begin{align}\label{eq:ovb_block_bound_completed_2}
\|(\bfB^{\mathrm{ovb}})_{j,j'}\|
\lesssim
R' 2^{jr_1+j's'-J(j')(r_1-s)}
\sqrt{\frac{2^{J(j')d}+\log ((\log N)/\delta)}{N}} .
\end{align}
Indeed, since the population cross-covariance vanishes, the OVB contribution is
controlled by the empirical fluctuation
\[
\left\|
(\widehat{\bar\Sigma}-\bar\Sigma)_{ I_{j'},I_{j'}^c}
\bfD^{-(r_1-s)}
\right\|=\left\|
(\widehat{\bar\Sigma})_{ I_{j'},I_{j'}^c}
\bfD^{-(r_1-s)}
\right\|.
\]
By the sample cross-covariance bound, uniformly over \(j'\),
\begin{align*}
\left\|
(\widehat{\bar\Sigma}-\bar\Sigma)_{ I_{j'},I_{j'}^c}
\bfD^{-(r_1-s)}
\right\|
&\lesssim
2^{-J(j')(r_1-s)}
\left(
\sqrt{\frac{2^{J(j')d}+\log((\log N)/\delta)}{N}}
+
\frac{2^{J(j')d}+\log((\log N)/\delta)}{N}
\right).
\end{align*}
Since \(2^{J(j')d}\lesssim N\) on \(\mcE_N\), the linear term is dominated by
the square-root term, and hence
\[
\left\|
(\widehat{\bar\Sigma})_{ I_{j'},I_{j'}^c}
\bfD^{-(r_1-s)}
\right\|
\lesssim
2^{-J(j')(r_1-s)}
\sqrt{\frac{2^{J(j')d}+\log((\log N)/\delta)}{N}}.
\]
Combining this estimate with the blockwise regularity bound for \(\bfA\)
gives \eqref{eq:ovb_block_bound_completed_2}.

It remains to check that, after applying the weights \(2^{-jt-j't'}\), the
leading deterministic factor in \eqref{eq:ovb_block_bound_completed_2} is of
order \(N^{-\gamma}\). Ignoring logarithmic factors, it suffices to prove
\[
2^{-jt-j't'}
2^{jr_1+j's'-J(j')(r_1-s)}
\sqrt{\frac{2^{J(j')d}}{N}}
\lesssim N^{-\gamma},
\]
or equivalently,
\[
2^{-J(j')(r_1-s)}
\sqrt{\frac{2^{J(j')d}}{N}}\,
2^{-jt-j't'+jr_1+j's'}
\lesssim N^{-\gamma}.
\]
Equivalently, it suffices to show
\[
2^{-J(j')(r_1-s-d/2)+j(r_1-t)-j'(t'-s')}
\lesssim N^{1/2-\gamma}.
\]
For \((j,j')\in\mcE_N\), we have
\[
j\le J(j')
=\frac{\gamma\log_2N-j'(t'-s')}{t-s}.
\]
Therefore
\[
\begin{aligned}
&-J(j')\Bigl(r_1-s-\frac d2\Bigr)+j(r_1-t)-j'(t'-s')\\
&\le
-J(j')\Bigl(r_1-s-\frac d2\Bigr)
+J(j')(r_1-t)_{+}
-j'(t'-s')\\
&=
J(j')\left((r_1-t)_{+}-r_1+s+\frac d2\right)
-j'(t'-s').
\end{aligned}
\]
If \(r_1<t\), then
\[
(r_1-t)_{+}-r_1+s+\frac d2
=
-r_1+s+\frac d2<0,
\]
by the assumption \(r_1-s>d/2\). Hence the exponent is nonpositive.

It remains to consider \(r_1\ge t\). In this case,
\[
(r_1-t)_{+}-r_1+s+\frac d2
=
\frac d2-(t-s).
\]
If \(t-s\ge d/2\), the exponent is again nonpositive. If \(t-s<d/2\), then
\[
\begin{aligned}
&-J(j')\Bigl(r_1-s-\frac d2\Bigr)+j(r_1-t)-j'(t'-s')\\
&\le
J(j')\Bigl(\frac d2-(t-s)\Bigr)\\
&\le
\gamma\log_2N
\left(\frac{d}{2(t-s)}-1\right)\\
&\le
\left(\frac12-\gamma\right)\log_2N,
\end{aligned}
\]
where the last inequality follows from
\[
\gamma\le \frac{t-s}{2(r_1-s)+d}\le \frac{t-s}{d}.
\]
This proves
\[
2^{-J(j')(r_1-s-d/2)+j(r_1-t)-j'(t'-s')}
\lesssim N^{1/2-\gamma},
\]
and hence the OVB contribution is bounded by \(N^{-\gamma}\), up to logarithmic
factors.

\end{document}